\documentclass[12pt,a4paper]{article}
\usepackage{amsmath, amsfonts, xifthen, latexsym, amssymb, amsthm, amscd, extarrows}
\usepackage{amsrefs,mdframed, etoolbox}
\usepackage{multirow}
\usepackage{tikz-cd}
\usepackage{authblk}

\usepackage[utf8]{inputenc}
\usepackage[T1]{fontenc}

\usepackage[margin=1.4in]{geometry}

\usepackage[shortlabels]{enumitem}
\usepackage{graphicx,xcolor}
\usepackage{url}
\usepackage{hyperref}
\hypersetup{colorlinks=true,citecolor=blue,filecolor=blue,linkcolor=blue,urlcolor=blue}
\usepackage[title]{appendix}
\usepackage[perpage]{footmisc}

\usepackage{tikz}
\usetikzlibrary{arrows}
\usepackage{subcaption}

\allowdisplaybreaks

\newtheorem{theorem}{Theorem}[section] 
\newtheorem{theorem*}{Theorem}

\newtheorem{lemma}[theorem]{Lemma}
\newtheorem{lemma*}[theorem*]{Lemma}
\newtheorem{claim}[theorem]{Claim}
\newtheorem{claim*}[theorem*]{Claim}

\newtheorem{prop}[theorem]{Proposition}
\newtheorem{prop*}[theorem*]{Proposition}

\newtheorem{corollary}[theorem]{Corollary}
\newtheorem{corollary*}[theorem*]{Corollary}

\theoremstyle{definition}
\newtheorem{definition}[theorem]{Definition}
\newtheorem{definition*}[theorem*]{Definition}
\newtheorem{question*}[theorem*]{Question}

\mdfdefinestyle{mdfex}{%
    linewidth=0.2em, 
    topline=false, 
    bottomline=false, 
    rightline=false,
    leftline=true
}

\newcommand{\newenv}[3]{%
    \newenvironment{#1}{%
        \begin{mdframed}[style=#3]
        \begin{#1internal}
    }{%
        \end{#1internal}
        \end{mdframed}
    }
}

\newenv{remark}{Remark}{mdfex}
\newenv{example}{Example}{mdfex}

\newcommand{\beq}[1][]{ 
    \ifthenelse{\isempty{#1}}{\begin{equation}}{\begin{equation}\label{#1}} 
}
\newcommand{\eeq}{\end{equation}}

\makeatletter
\newtheorem*{rep@theorem}{\rep@title}
\newcommand{\newreptheorem}[2]{%
\newenvironment{rep#1}[1]{%
 \def\rep@title{#2 \ref{##1}}%
 \begin{rep@theorem}}%
 {\end{rep@theorem}}}
\makeatother
\newreptheorem{theorem}{Theorem}
\newreptheorem{lemma}{Lemma}

\DeclareFontFamily{U}{mathx}{\hyphenchar\font45}
\DeclareFontShape{U}{mathx}{m}{n}{<-> mathx10}{}
\DeclareSymbolFont{mathx}{U}{mathx}{m}{n}
\DeclareMathAccent{\widebar}{0}{mathx}{"73}

\newcommand{\om}{{\omega}}

\renewcommand{\phi}{{\varphi}}

\newcommand{\one}{\mathbf{1}}

\newcommand\ZZ{\mathbb Z}
\newcommand\FF{\mathbb F}
\newcommand\RR{\mathbb R}

\newcommand\NN{\mathbb N}

\renewcommand\SS{\mathbb S}

\newcommand{\actson}{\curvearrowright}

\renewcommand{\cal}[1]{{\mathcal #1}}

\provideboolean{theno_parts}
\setboolean{theno_parts}{true}
\renewcommand{\part}[1][]{%
    \ifthenelse{\boolean{theno_parts}}{%
        \begin{enumerate}[wide,label=(\alph*),itemsep=1pt,topsep=0pt,#1]%
            \item
    }{%
        \item
    }
    \setboolean{theno_parts}{false}%
}
\newcommand{\trap}{%
    \ifthenelse{\boolean{theno_parts}}{}{%
        \end{enumerate}
    }
    \setboolean{theno_parts}{true}
}

\newcommand{\newop}[2]{%
    \expandafter\def\csname #1\endcsname{\operatorname{#2}}
}

\renewcommand\tilde{\widetilde}
\newcommand\Rel{\operatorname{Rel}}

\newop{Pow}{Pow}
\newop{im}{im}

\newop{subexp}{SubExp}

\newop{Pr}{\mathbb{P}}
\newop{E}{\mathbb{E}}
\newop{maxoutdeg}{max-outdeg}
\newop{maxdeg}{max-deg}
\newop{res}{res}

\newop{free}{free}
\newop{Good}{Good}
\newop{Spc}{Space}
\newop{Gal}{Gal}
\newop{Z}{Z}
\newop{SL}{SL}
\newop{SO}{SO}
\newop{cost}{cost}
\newop{Vol}{Vol}
\newop{Res}{Res}
\newop{Side}{Side}
\newop{refin}{refine}
\newop{ends}{ends}
\newop{levels}{levels}
\newop{Lifts}{Lifts}
\newop{len}{len}
\newop{Torso}{Torso}
\newop{Cyclic}{Cyclic}
\newop{Check}{Check}
\newop{maxlen}{maxlen}
\newop{inside}{inside}
\newop{outside}{outside}
\newop{Star}{Star}
\newop{rot}{rot}
\newop{col}{col}
\newop{proj}{proj}
\newop{edit}{edit}
\newop{graph}{graph}
\newop{Sides}{Sides}
\newop{cycs}{CycSides}
\newop{faces}{Faces}
\newop{espan}{span}
\newop{Bij}{Bij}
\newop{Cay}{Cay}
\newop{tw}{tw}
\newop{inte}{int}
\newop{Classes}{Classes}
\newop{Isom}{Isom}
\newop{Rep}{Rep}
\newop{levseq}{LevSeq}
\newop{comp}{Comp}
\newop{ClaRep}{ClaRep}
\newop{Inv}{Inv}
\newop{Aut}{Aut}
\newop{Adj}{Adj}
\newop{Seq}{Seq}
\newop{Sep}{Sep}
\newop{vx}{Vx}
\newop{cut}{Cut}
\newop{Fin}{Fin}
\newop{Thin}{Thin}
\newop{Bor}{Bor}

\title{Applications of tree decompositions and accessibility to treeability of Borel graphs}
\author{H\'ector Jard\'on-S\'anchez \thanks{\href{mailto:sanchez@math.uni-leipzig.de}{sanchez@math.uni-leipzig.de}}}
\affil{University of Leipzig}

\begin{document}
\maketitle
\begin{abstract}
A framework to handle tree decompositions of the components of a Borel graph in a Borel fashion is introduced, along the lines of Tserunyan's Stallings Theorem for equivalence relations \cite{Tse}. This setting leads to a notion of accessibility for Borel graphs, together with a treeability criterion. This criterion is applied to show that, in particular, Borel equivalence relations associated to Borel graphs with accessible planar connected components are measure treeable, generalising results of Conley, Gaboriau, Marks, and Tucker-Drob \cite{CGMT} and Timar \cite{Timar}. It is also proven that uniformly locally finite Borel graphs with components of bounded tree-width yield Borel treeable equivalence relations. Our results imply that p.m.p~countable Borel equivalence relations with measured property (T) do not admit locally finite graphings with planar components a.s.
\end{abstract} 
\tableofcontents

\section{Introduction}

Let $X$ be a standard Borel space and let $\cal G \subset X\times X$ be a Borel graph of bounded vertex degree. We recall that $\cal G$ induces a countable Borel equivalence relation $\Rel (\cal G)$ on its vertex space $X$ defined by letting $x\sim y$ for $x,y \in X$ if and only if $x$ and $y$ are in the same connected component of $\cal G$. Often, we will refer to connected components simply as \emph{components}.

Following \cite{CGMT}, we say that $\Rel (\cal G)$ is \emph{Borel treeable} if there exists an acyclic Borel graph $\cal T$ on $X$ such that $ \Rel (\cal G) = \Rel (\cal T)$. If $\mu$ is a Borel probability measure on $X$, we say that $\Rel (\cal G)$ is $\mu$\emph{-treeable} if there exists a Borel set $X_0 \subset X$ with $\mu (X_0) = 1$ such that the restriction $\Rel (\cal G)|_{X_0}$ is Borel treeable. Finally, $\Rel (\cal G)$ is \emph{measure treeable} if it is $\mu$-treeable for every Borel probability $\mu$ on $X$.

We recall that a graph is \emph{planar} if it is embeddable in $\RR^2$. The starting point of this paper is the following question. 

\begin{question*}\label{q1}
If $\cal G$ is a Borel graph with planar components, is $\Rel (\cal G)$ measure treeable? 
\end{question*}


Treeability has various applications in the theory of p.m.p.~Borel equivalence relations. Recall that a \emph{p.m.p.~countable Borel equivalence relation} is a pair $(R,\mu)$ such that $R$ is a countable Borel equivalence relation and $\mu$ is a Borel probability measure on the vertex space $X$ of $R$ such that $f_* \mu = \mu$ for every Borel bijection $f\colon X\rightarrow X$ such that $\graph (f) \subset R$. A \emph{graphing} of such $(R,\mu)$ is a Borel graph $\cal G$ such that $\Rel (\cal G) = R$ a.s. 

On the one hand, the cost of a treeable p.m.p.~Borel equivalence relation is attained, and hence can be computed, by its treeings \cite{Gab}. On the other hand, treeable p.m.p.~Borel equivalence relations are sofic \cite{EL}. This implies that graphings of treeable p.m.p.~Borel equivalence relations satisfy the Aldous-Lyons conjecture \cite{AL}. One of our motivations to study Question \ref{q1} is the open question of whether graphings with planar components satisfy the Aldous-Lyons conjecture.

 Here we will not give a full answer to Question \ref{q1}. The most general answer here offered to Question \ref{q1} is Theorem \ref{tma:planar}. Before we address its statement, let us present an implication of it. Following Thomassen and Woess \cite{Thom93}, we recall that a graph $G$ is \emph{(edge)-accessible} if there exists $k \in \NN$ such that every pair of distinct ends of $G$ can be separated by removing at most $k$ edges of $G$.

\begin{theorem*}[See Theorem \ref{tma:planar}]\label{tma:introplanac}
Let $\cal G$ be a locally finite Borel graph with planar accessible components. Then $\Rel (\cal G)$ is measure treeable.
\end{theorem*}

We remark that the above theorem is a strict corollary of Theorem \ref{tma:planar}. In Example \ref{ex:freepakosanz} we provide an example of locally finite Borel graph $\cal G$ together with an invariant measure such that components of $\cal G$ are planar and not accessible a.s. Still $\cal G$ falls under the hypothesis of Theorem \ref{tma:planar}.

Let us now discuss how Theorem \ref{tma:introplanac} relate to previous results about Borel graphs with planar components in the literature. Conley, Gaboriau, Marks, and Tucker-Drob~\cite{CGMT}, proved that locally finite Borel graphs admitting a Borel 2-basis are measure treeable. By a theorem of Thomassen \cite{Thom80}, \emph{2-bases} of planar graphs correspond to sets of facial cycles in accumulation-free embeddings in $\RR^2$. A \emph{Borel 2-basis} for a Borel graph $\cal G$ is a 2-basis $\cal B$ of $\cal G$ which is Borel as a subset of the standard Borel space of finite subsets of $\cal G$. 

Example \ref{ex:adams} is an example of uniformly locally finite Borel graph with planar accessible components not admitting a 2-basis. This shows that Theorem 2 applies in a strictly more general context than the results of \cite{CGMT}. 

In \cite{Timar}, Timar showed that if $(\cal G,\mu)$ is a graphing with one-ended planar components and finite expected degree, then $\cal G$ is $\mu$-treeable. One-ended graphs satisfy the definition of accessibility vacuously, so this result is implied by Theorem \ref{tma:introplanac}. 

Our strategy to prove Theorem \ref{tma:introplanac} is, informally, to use tree decompositions to break down Borel graphs into pieces which we know how to tree, in the planar case, by the results of \cite{CGMT}. Then, the treeings in these pieces are glued back to obtain a treeing of the original Borel graph. This strategy applies in further generality and so in this paper, specially in Section 4 we offer a general treatment to it. Before we discuss it more precisely, let us just state another corollary of it.

Following \cite{Diestel05} we recall that two rays in a graph $G$ are \emph{equivalent} if for every finite $F \subset V(G)$ there are tails of both rays in the same connected component of $G \backslash F$. An \emph{end} of $G$ is an equivalence class of rays. An end $\omega$ is \emph{thick} if there exist infinitely many disjoint rays in $\omega$. 

Tree decompositions have a dedicated section in the preliminaries. Roughly speaking, a tree decomposition of a connected graph $G$ is a pair $(T,\cal V)$ such that $T$ is a tree and $\cal V = (V_t)_{t\in V(T)}$ is a family of vertex subset with $V_t \subset V(G)$ for every $t\in V(T)$ such that $T$ rules how the $V_t$ are arranged in $G$ and $G = \bigcup_{t\in V(T)} G[V_t]$.  

The \emph{width} of a tree decomposition $(T,\cal V)$ of $G$ is $\sup \{|V_t|-1 : t\in V(T)\}$. The \emph{tree-width} $\tw (G)$ of $G$ is the infimum of the tree-width of its tree decompositions. We say that a graph has \emph{bounded tree-width} if $\tw (G) < \infty$. 

\begin{theorem*}[See Corollary \ref{cor:bddtw}]\label{tma:bddtwintro}
Borel graphs with no thick ends and accessible components are Borel treeable. In particular, uniformly locally finite Borel graphs with bounded tree-width components are Borel treeable.
\end{theorem*}

Moreover, our arguments show that one-ended Borel graphs with bounded tree-width components are, moreover, hyperfinite (see Corollary \ref{cor:hyper}).

Borel treeability for locally finite Borel graphs with bounded tree-width components has been proved, using independent methods, by  Ronnie Chen, Antoine Poulin, Ran Tao, and Anush Tserunyan. 

Let us now begin an outline of the paper's structure, introducing along the way our results in further generality. 

Section 2 is devoted to preliminaries. In Section 3, we adapt the graph-theorical theory of tree decompositions to the setting of Borel graphs. Let us fix a locally finite Borel graph $\cal G \subset X\times X$ where $X$ is a standard Borel space. We introduce \emph{nested Borel separation systems} to encode, in a Borel fashion, the information required to induce tree decompositions of the components of a Borel graph. Under the further assumption of having \emph{property ($\ast$)}, introduced in \cite{CHM}, such a nested Borel separation system $\cal S$ induces a tree decomposition $(\cal T_G , \cal V_G)$ on each component $G$ of $\cal G$. 

If $\cal G$ is a Borel graph and $\cal S$ a nested Borel separation system with property ($\ast$), then one can construct an auxiliary standard Borel space $Z$ with elements of the form $(x,t)$ where $x\in X$ and $t\in V(T)$ is such that $x\in V_t$. We define the \emph{levels equvialence relation associated to $\cal S$} to be the Borel equivalence relation $Q$ on $Z$ defined by $(x,t) \sim_Q (y,t')$ if and only if $t = t'$. We refer to the classes of $Q$ as \emph{levels}. Then, one can state the main result of Section 3. The measure $\mu_Z$ mentioned in the theorem below is just an auxiliary measure on $Z$ constructed in Section 3 related to a Borel measure $\mu$ on $X$.

\begin{theorem*}[See Theorem \ref{tma:maintech}]\label{tma:maintechintro}
Let $\cal G$ be a locally finite Borel graph on a standard Borel space $X$, let $\cal S$ be a nested Borel separation system with property ($\ast$), and let $\mu$ be a Borel probability measure on $X$. If $Q$ is Borel treeable (resp.~$\mu_Z$-treeable), then $\Rel (\cal G)$ is Borel treeable (resp.~$\mu$-treeable). In particular, if $Q$ is measure treeable, then $\Rel (\cal G)$ is measure treeable.
\end{theorem*} 

Let us now remark that the above theorem is very closely related to Tserunyan's Theorem \cite[Corollary 4.20]{Tse}. Regarding Borel treeability of multi-ended Borel graphs, both results yield the same conclusions. However, our theorem includes the cases of $\mu$-treeability and measure treeability and drops the multi-ended condition. Beyond this, the framework of the auxiliary space $Z$ and the levels equivalence relation on it becomes helpful when giving a graph-based approach to the application of tree decompositions in the Borel setting.

In Section 4, we make precise the rough idea of a Borel graph admitting tree decompositions whose levels can be graphed with certain desired combinatorial properties. This is encoded in the following definition.

\begin{definition*}[See Definition \ref{def:boracc}]
Let $\cal F$ be a family of locally finite connected graphs. We say that a locally finite Borel graph $\cal G$ is \emph{Borel $\cal F$-accessible} if there exists a nested Borel separation system $\cal S \subset \Sep (\cal G)$ with property $(\ast)$ and a Borel graph $\cal H $ with $\Rel (\cal H) = Q$, where $Q$ is the levels equivalence relation associated to $\cal S$, such that $\cal H$-components are in $\cal F$.
\end{definition*}

The motivation to study Borel $\cal F$-accessibility comes from the following straightforward corollary of Theorem \ref{tma:maintechintro}. This theorem encapsulates the general strategy to prove treeability portrayed in this paper.

\begin{theorem*}\label{tma:Faccessibleintro}
Let $\cal F$ be a family of locally finite connected graphs. If every  Borel graph with components in $\cal F$ is Borel (resp.~measure) treeable, then so is any Borel $\cal F$-accessible locally finite graph. 
\end{theorem*}

The use of the term accessibility in the definition of Borel $\cal F$-accessibility is justified by the following result. Let $\cal F_1 \subset \cal F$ denote the subfamily of one-ended graphs in $\cal F$. 

\begin{theorem*}\label{tma:accesintro}
Let $\cal F$ be a minor-closed family of locally finite connected graphs. Then, Borel graphs with all components accessible and in $\cal F$ are Borel $\cal F_1$-accessible. 
\end{theorem*}

Section 4 concludes with an application of the two above theorems. Let $\Thin$ denote the class of locally finite connected graphs with no thick ends. Such family is minor-closed by Halin's Theorem \cite{Hal}. In Theorem \ref{tma:thin} we prove that Borel graphs with components in $\Thin_1$ are Borel treeable. Then, it follows automatically from Theorems \ref{tma:Faccessibleintro} and \ref{tma:accesintro} that Borel $\Thin_1$-accessible Borel graphs are Borel treeable. In particular we obtain Theorem \ref{tma:bddtwintro} above.

By a similar argument it follows that in order to prove Theorem \ref{tma:introplanac} it suffices to show the following. 

\begin{theorem*}[Theorem \ref{tma:one-end-plan}]\label{tma:one-end-plan-intro}
Let $\cal G$ be a locally finite Borel graph with planar components. If $\cal G$ is at most one-ended, then $\cal G$ is measure treeable.
\end{theorem*}

In particular, we can also now state our most general partial answer to Question \ref{q1} above. Let $\cal P$ denote the class of locally finite connected planar graphs. Then, as a consequence again of Theorems \ref{tma:Faccessibleintro} and \ref{tma:accesintro} we have the following.

\begin{theorem*}[Theorem \ref{tma:planar}]\label{tma:planarintro}
Borel $\cal P_1$-accessible Borel graphs are measure treeable. 
\end{theorem*}

Section 5 is devoted to the proof of Theorem \ref{tma:one-end-plan-intro}. This result is proven as an implication of the measure treeability result for Borel graphs with a Borel 2-basis in \cite{CGMT}. However, the application of the latter is not straightforward. Whilst we know that one-ended planar graphs admit a 2-basis, we do not know \emph{a priori} whether Borel graphs with one-ended planar components admit a Borel 2-basis or not. It is not known in full generality whether a Borel graph whose components admit a 2-basis has a Borel 2-basis. Here, we prove the following partial answer to the question.

\begin{corollary*}[Corollary \ref{cor:dosbase}]\label{cor:dosbaseintro}
Let $\cal G$ be a locally finite Borel graph with 3-connected planar components admitting a 2-basis. Then, $\cal G$ admits a Borel 2-basis.
\end{corollary*}

We now point out the tools that will be developed in order to prove the above corollary. Informally, in order to take advantage of the planarity of the components of a Borel graph $\cal G$ we need to endow $\cal G$ with some additional structure encoding the information about some embedding of the components into either the plane $\RR^2$ or the 2-dimensional sphere $\SS$. In this article we encode planarity using rotation systems. 

A \emph{spherical rotation system} on a locally finite graph $G$ is a collection $\omega = \{\omega_v\}_{v\in V(G)}$ such that $\omega_v$ is a cyclic order on the neighbours of $v \in V(G)$ and every $\omega_v$ is induced by the same embedding of $G$ in $\SS$. A spherical rotation system is obtained by embedding the graph $G$ in $\SS$ and letting each vertex record the order in which its neighbours are arranged. In this sense, a spherical rotation system may be regarded as a ``combinatorial embedding'' of a graph.

In Section 5 we introduce and study Borel graphs with planar components for which a spherical rotation system has been specified in a Borel fashion. More precisely, we introduce the following class of Borel graphs. In the following definition, we denote by $\cal G_x$ the countable, connected graph obtained as the component of the vertex $x$ in $\cal G$.

\begin{definition*}[See Definition \ref{def:borsph}]
A locally finite Borel graph $\cal G$ on a standard Borel space $X$ is \emph{Borel spherical} if $\cal G$ is equipped with a Borel map $\omega$ with domain $X$ such that for every $x\in X$, $\om(x)$ is a cyclic order on the neighbourhood of $x$, and  $\{\om (x)\}_{x \in V(\cal G_x)}$ is a spherical rotation system on $\cal G_{x}$.  Such a map $\omega$ is called a \emph{Borel rotation system}.
\end{definition*} 

In Example \ref{ex:adams} we provide an example showing that not all Borel graphs with spherical components are Borel spherical. In spite of this, in the case of graphs with 3-connected components, one can still prove the following proposition. Recall that an \emph{extension} of a Borel graph $\cal G$ is a Borel graph homomorphism $f\colon \cal H \rightarrow \cal G$, where $\cal H$ is a Borel graph, such that $f$ restricted to each component of $\cal H$ is a graph isomorphism onto its image.

\begin{prop*} [Proposition \ref{prop:ext}]
Let $\cal G$ be a locally finite Borel graph with 3-connected spherical components. Then, there exist a Borel spherical graph $\cal H$ and a 2-to-1 extension $\cal H \rightarrow \cal G$. 
\end{prop*}

The above proposition is the main ingredient in the proof of Corollary \ref{cor:dosbaseintro}. The author believes that a Borel spherical extension may be obtained for any Borel graph with planar components though it may not be 2-to-1 in general. However, such a result is not proven here. 

In the last section of this paper, namely Section 6, we show that, despite of not having proven measure treeability of arbitrary locally finite Borel graphs with planar components, we can rule out measured property (T) for equivalence relations induced by Borel graphs with planar components. Before we state this result we recall some definitions.

Recall that a \emph{p.m.p.~extension} of p.m.p.~countable Borel equivalence relation $(R,\mu)$ is another p.m.p.~countable Borel equivalence relation $(Q,\nu)$ together with a Borel map $f\colon Y \rightarrow X$, where $Y$ and $X$ are the vertex spaces of $Q$ and $R$ respectively, such that $f_* \nu = \mu$ and $y \sim_Q y'$ implies $f(y) \sim_R f(y')$ for almost every $y,y'\in Y$.

A p.m.p.~countable Borel equivalence relation $(R,\mu)$ is \emph{approximable} if there exists an increasing sequence of Borel subequivalence relations $(R_n)_{n\in \NN}$ such that $R = \bigcup_{n\in \NN} R_n$ and, for any Borel subset $A$ of the vertex space, we have that $R_n|_A = R|_A$ for some $n$ implies $\mu (A) = 0$. We say that $(R,\mu)$ has \emph{measured property (T)} if and only if no p.m.p.~extension of $(R,\mu)$ is approximable. This definition is equivalent to that introduced by Moore \cite{Moo} and Zimmer \cite{Zim}. This follows from soon available work of the author with \L{}ukasz Grabowski and Samuel Mellick. 

\begin{theorem*}\label{tma:noTintro}
Let $(R,\mu)$ be a p.m.p.~Borel equivalence relation with measured property (T). Then $R$ does not admit a locally finite graphing with planar components a.s.
\end{theorem*}

Let us briefly outline the proof of this theorem. First, in Proposition \ref{prop:multi} we show that  p.m.p.~Borel equivalence relations with a.s.~multi-ended graphings are approximable. This proof uses techniques involving Borel separation systems. This Proposition implies that graphings of p.m.p.~Borel equivalence relation with measured property (T) must be one-ended. In combination with Theorem \ref{tma:planarintro}, we deduce Theorem \ref{tma:noTintro}.

To conclude this introduction we pose a question. By \cite{Car}, every locally finite connected graph admits a nested separation system with property ($\ast$) separating any two ends. Theorem \ref{tma:accesintro} above proves that every locally finite Borel graph with accessible components admits a  nested Borel  separation system with property ($\ast$) separating any two ends in the same component. 

\begin{question*}[Question \ref{q2}]\label{q2intro}
For a locally finite borel graph $\cal G$, does there exist a nested Borel  separation system with property ($\ast$) separating any two ends in the same component?
\end{question*}

As discussed in Remark \ref{rem:q2}, a positive answer to the above question would allow us to drop the accessibility hypothesis in Theorem \ref{tma:accesintro}. In this situation, the accessibility hypotheses of Theorems \ref{tma:bddtwintro} and \ref{tma:planarintro} could be dropped. Namely, a positive answer to Question \ref{q2intro} would imply that locally finite Borel graphs with planar components are measure treeable, and that locally finite Borel graphs with no thick ends are Borel treeable.

\section{Preliminaries}

\subsection{Graphs}

The reader is referred to \cite{Diestel05} for further details on the fundamentals of graph theory presented here. In this paper we consider undirected, simple, locally finite graphs with countable connected components. We will often refer to the connected components of a graph as, simply, components. The vertex set of a graph $G$ is denoted by $V(G)$ and its edge set by $E(G)$. The \emph{neighbourhood} of a vertex $v\in V(G)$ is denoted by $N_G (v)$. 

The graph distance on $G$ is denoted by $d_G$. The \emph{ball} of radius $r \in \NN$ around a vertex $v \in V(G)$, denoted by $B_r (v)$, is the subgraph of $G$ induced by $\{u \in V(G) \colon d_G (u,v) \leq r\}$.  More generally, if $A \subset V(G)$, the induced subgraph on $A$ is denoted by $G[A]$, and its \emph{boundary} $\partial_G A$ is the subset of edges $e \in E(G)$ with one end-vertex in $A$ and the other end-vertex in $V(G) \backslash A$. In a slight abuse of notation, we denote by $G \backslash A$ the graph $G[V(G) \backslash A]$. Similarly, if $B\subset E(G)$, we let $G \backslash B$ denote the graph with vertex set $V(G)$ and edge set $E(G)\backslash E$.

A \emph{path} is a graph $P$ with vertex set $V(P) = \{x_1,\dots,x_n\}$ and edge set $E(P) = \{(x_i,x_{i+1}):i=1,\dots,n-1\}$. Given such a path $P$, a graph $C$ with $V(C) = V(P)$ and $E(C) = E(P) \cup \{(x_n,x_1)\}$ is a \emph{cycle}. A \emph{ray} is a semi-infinite path, i.e.: a graph $R$ with $V(R) = \{x_i : i\in \NN\}$ and $E(R) = \{(x_i,x_{i+1}): i\in \NN\}$. A \emph{tail} of a ray is an infinite connected subgraph of the ray. By paths, cycles, and rays in a graph $G$ we refer to subgraphs which are, respectively, paths, rays, or cycles.

 Two rays in a graph $G$ are \emph{equivalent} if for every finite $F \subset V(G)$ there are tails of both rays in the same connected component of $G \backslash F$. An \emph{end} of $G$ is an equivalence class of rays. The set of ends of a graph $G$ is denoted by $\ends (G)$. An end $\omega$ is \emph{thick} if there exist infinitely many disjoint rays in $\omega$, and \emph{thin} otherwise.

A graph $G$ is $k$\emph{-(edge)-accessible}, for $k \in \NN$, if for every two distinct $\omega,\rho \in \ends (G)$ there exists $B \subset E(G)$ with $|B| \leq k$ such that any two rays $R_\omega$ and $R_\rho$ representing $\omega$ and $\rho$ respectively have tails in different connected components of $G\backslash B$. We say that a graph $G$ is \emph{accessible} if it is $k$-accessible for some $k\in \NN$. 

A graph $G$ is $k$\emph{-vertex-accessible}, for $k \in \NN$, if for every two distinct $\omega,\rho \in \ends (G)$ there exists $U \subset V(G)$ with $|B| \leq k$ such that any two rays $R_\omega$ and $R_\rho$ representing $\omega$ and $\rho$ respectively have tails in different connected components of $G\backslash U$. We say that a graph $G$ is \emph{vertex-accessible} if it is $k$-vertex-accessible for some $k\in \NN$.

A graph $G$ is $k$\emph{-connected} if for every set $A \subset V(G)$ with $|A| < k$, we have that $G \backslash A $ is connected, i.e.~$G$ stays connected whenever we remove less than $k$ vertices. Alternatively, we may characterise vertex connectivity in terms of paths. The proof of the following theorem is a straightforward adaptation of the proof of~\cite[Proposition 8.4.1]{Diestel05}.

\begin{theorem}[Menger]
A graph $G$ is $k$-connected if and only if for every pair $u,v \in V(G)$ of distinct vertices there exist at least $k$ internally disjoint paths with $u$ and $v$ as their end-vertices.
\end{theorem}

Following \cite{RST}, given two graphs $G$ and $H$, we say that $G$ has an $H$\emph{-minor} if for each $v\in V(H)$ there exists a non-empty connected subgraph $\alpha_v \subset G$ and for each $e\in E(H)$ there exists an edge $\alpha_e \in E(G)$ such that 
\part for each distinct $u,v\in V(H)$ the graphs $\alpha_u$ and $\alpha_v$ are vertex-disjoint,

\part for each distinct $e,f\in E(H)$ we have $\alpha_e \neq \alpha_f$,

\part for each $v\in V(H)$ and $e\in E(H)$, we have $\alpha_e \not\in E(\alpha_v)$, and

\part if $(u,v) \in E(H)$, then the end-vertices of $\alpha_{(u,v)}$ are contained in $\alpha_u$ and $\alpha_v$.
\trap

\subsection{Tree decompositions}

A \emph{tree decomposition} of a connected graph $G$ is a pair $(T, \cal V)$ where $T$ is a tree and $\cal V = (V_t)_{t \in V(T)}$ is a collection of subsets of $V(G)$, called the \emph{parts}, such that 

\part $V(G) = \bigcup_{t \in V(T)} V_t$,

\part for every $(u,v) \in E(G)$ there exists $t \in V(T)$ such that $\{u,v\}\subset V_t$, and

\part if $t_1, t_2, t_3 \in V(T)$ are such that $t_2$ lies in the path between $t_1$ and $t_3$, then $V_{t_1} \cap V_{t_3} \subset V_{t_2}$.

\trap
In a slight abuse of notation, we will refer to the vertices or edges of $T$ as the vertices or edges of the tree decomposition.  The \emph{torsos} of $(T,\cal V)$ are the graphs $H_t$ obtained by adding to the induced graph $G[V_t]$ an edge $(x,y)$ if $x,y \in V_t \cap V_{t'}$ for some neighbour $t' \in V(T)$ of $t$. We refer to the edges in $E(H_t) \backslash E(G[V_t])$ as \emph{virtual edges}.


%

A \emph{separation} of a countable connected graph $G$ is a pair  $(S,B)$ where $S$ is a finite subset of $V(G)$ and $B \subset \partial S$ satisfies the following property. Suppose that $(u,v) \in \partial S$, where $u\in S$ and $v\not \in S$, is such that there exists $(u',v')\in B$ with $u'\in S$ and $v'\not \in S$ satisfying that $v$ and $v'$ lie in the same connected component of $G\backslash S$. Then we require that $(u,v) \in B$. The set of separations of a graph $G$ is denoted by $\Sep (G)$.

We say that a separation is \emph{proper} if $G\backslash S$ is not connected and $\emptyset \neq B \subsetneq \partial S$. For a proper separation, the choice of $B$ is equivalent to the  choice of a non-empty proper subset of the set of connected components of $G\backslash S$.

The \emph{adhesion set} of a separation $(S,B)$ is the set $S$. The \emph{order} of a separation $(S, B)$ is $|S|$. The \emph{cut} of a separation $(S,B)$ is the set $B$. If $(S,B)$ is a separation, we let $(S,B)^- := (S, (\partial S) \backslash B)$.

Let $(S,B)$ be a separation and let 
$$
    C:=\{u \in V(G) \backslash S\colon  \exists v \in S \text{  such that }(u,v) \in B\}.
$$ 
Let $\overline{C}$ be the set of vertices $w\in V(G) \backslash S$ for which there exists some $u \in C$ such that $u$ and $w$ belong to the same $G\backslash S$-connected component. The \emph{sides} of $(S,B)$ are the graphs $A_1= G[S \cup \overline{C}]$ and $A_2 = G[V(G) \backslash \overline{C}]$. We let $\Sides (S,B)$ denote the ordered pair $(A_1,A_2)$. We note that $\Sides((S,B)^-) = (A_2,A_1)$

\begin{remark}
If $(S,B)$ is a separation and $(A_1,A_2) = \Sides (S,B)$, then $A_1 \cup A_2 = G$, so the ordered pair $(V(A_1),V(A_2))$ is what is usually referred to as a separation in the literature: an ordered pair $(U_1,U_2)$ with $U_1,U_2 \subset V(G)$ such that $G = G[U_1] \cup G[U_2]$. Conversely, given such an ordered pair $(U_1, U_2)$ we can construct a separation $(S,B):=( U_1 \cap U_2,  \partial S \cap E(G[U_1]))$.  It is straightforward to check, using the assumption that $G$ is connected, that $\Sides (S,B) = (G[U_1],G[U_2])$.

Hence, both definitions of separation are equivalent when $G$ is a connected graph. Under the correspondence established in the above paragraph, our definition of proper separation and the usual one in the literature -- a separation $(U_1,U_2)$ is proper if and only if $U_1 \backslash U_2$ and $ U_2 \backslash U_1$ are non-empty -- agree too. We opt for our definitions as they only require local information, which will become convenient when we deal with separations in the Borel setting in Section 3.
\end{remark}

The following is a reformulation of \cite[Lemma 12.3.1]{Diestel05} to our setting.

\begin{lemma}\label{lem:diesteltree}
Let $(T,\cal V)$ be a tree decomposition of a connected graph $G$. Consider an edge $(t_1,t_2)\in E(T)$, let $S := V_{t_1} \cap V_{t_2}$ and let $B_1$ (resp. $B_2$) be the subset of $\partial_G S$ consisting of edges with one end-vertex in $S$ and the other in some $V_t$, where $t \in V(T)$ lies in the same connected component of $T \backslash \{(t_1, t_2)\}$ as $t_1$ (resp. $t_2$). Then, the pairs $(S,B_1)$ and $(S,B_2)$ are separations of $G$.\qed
\end{lemma}

The separations $(S,B_1)$ and $(S,B_2)$ in the previous lemma are said to be \emph{induced} by the edge $(t_1,t_2)$.

The following constructions and definitions are extracted from \cite{CHM} and adapted to our definition of separation. The set of separations of $G$ admits a partial order $\leq$ defined as follows. Let $(S, B)$ and $(S', B')$ be separations of $G$ with sides $(A_1,A_2)$ and $(A'_1, A'_2)$ respectively. Then,
\begin{equation*}
(S,B) \leq (S',B') \xLeftrightarrow{\text{def}} A_1 \subset A'_1 \ \text{and}\ A'_2 \subset A_2.
\end{equation*}
Any two separations $(S, B)$ and $(S',B')$ induced by an edge of a tree decomposition are \emph{nested}, i.e.: $(S,B)$ is $\leq$-comparable with either $(S', B')$ or $(S', B')^-$.  

Conversely, let us discuss how certain families of separations induce tree decompositions. A family $\cal S$ of separations on $G$ is \emph{symmetric} if $(S,B) \in \cal S$ implies $(S,B)^- \in \cal S$, and \emph{nested} if any two separations in $\cal S$ are nested. We say, following \cite{CDHS}, that $\cal S$ is a \emph{separation system} if it is symmetric and every separation in $\cal S$ is proper. In particular, the family of separations induced by the edges of a tree decomposition is a nested separation system. The separation system \emph{induced} by a family of proper separations is the smallest separation system containing the family. It is routine to check that the sepation system induced by a nested family of separations is nested.

\begin{definition}[Carmesin-Hamann-Miraftab \cite{CHM}]
A nested separation system $\cal S$ has \emph{property} ($\ast$) if for every $(S, B), (S', B') \in \cal S$ such that $(S, B) < (S', B')$ there exist finitely many $(S'', B'') \in \cal S$ such that  $(S, B) < (S'', B'')< (S', B')$.
\end{definition}

Recall that, if $(A, \leq)$ is a poset, then $x \in A$ is a \emph{predecessor} of $y \in A$ if and only if $x < y $ and there exists no $z \in A$ such that $x < z < y$. We define an equivalence relation $\sim_{\cal S}$ on a nested separation system  $\cal S$ by letting $(S,B) \sim_{\cal S} (S', B')$ if and only if either $(S, B) = (S', B')$ or $(S,B)^-$ is a predecessor of $(S', B')$. We define a graph $T_{\cal S}$  whose vertices are the equivalence classes of $\sim_{\cal S}$, and edges are the pairs of the form $([(S, B)]_{\sim_{\cal S}} , [(S, B)^-]_{\sim_{\cal S}})$, where $(S, B)\in \cal S$. If we then let $\cal V_{\cal S}$ denote the family of all sets of the form $\bigcap \{A_1 \colon (A_1, A_2) = \Sides (S,B), (S,B)\in \chi\}$, where $\chi \in V(T_{\cal S})$, then we have the following.

\begin{prop}[Carmesin-Hamann-Miraftab \cite{CHM}]\label{prop:CHM}
Let $G$ be a connected graph and $\cal S$ a nested separation system with property $(\ast)$. Then 

\part $(T_{\cal S}, \cal V_{\cal S})$ is a tree decomposition of $G$, and 

\part the family of separations induced by the edges of $T_{\cal S}$ is equal to $\cal S$.
\trap
\end{prop}

The following is a criterion for a nested separation system to have property ($\ast$). Its proof is essentially that of \cite[Lemma 2.3] {CHM}

\begin{prop}\label{prop:estar}
Let $\cal S$ be a nested separation system on a countable connected graph $G$. Suppose that there exists $k \in \NN$ such that every separation in $\cal S$ has order at most $k$. Then $\cal S$ has property ($\ast$).
\end{prop}

\begin{example}
A \emph{cutvertex} of $G$ is a vertex $v \in V(G)$ such that $G \backslash \{v\}$ is not connected. Let $\cal S$ be separation system induced by all separations of order 1 on $G$. Then  $\cal S$ is a nested separation system with property ($\ast$) by Proposition \ref{prop:estar}. It is easy to see that the torsos of $(T_{\cal S}, \cal V_{\cal S})$ have no cutvertices. Therefore, these torsos are either single edges or \emph{blocks} of $G$: maximal 2-connected subgraphs. The block of a vertex $v \in V(G)$ which is not a cutvertex is denoted by $[v]_G$. 
\end{example}

\begin{example}\label{ex:Tutte}
Similarly as in the above example we may consider the \emph{Tutte decomposition} of a 2-connected graph $G$. The Tutte decomposition was introduced in \cite{Tutte} for finite graphs, extended to locally finite countable graphs in \cite{DSS}, and arbitrary countable graphs in \cite{Rich04}. The definition used here follows the lines of that in \cite{Diestel05}. 

The Tutte decomposition is the tree decomposition obtained from the separation system $\cal S$ induced by the family of all separations of order 2 which are nested with every other 2-separation. By Proposition \ref{prop:estar}, the set $\cal S$ is nested and has property ($\ast$). The arguments in \cite[Theorems 6 and 10]{Rich04} can be easily translated to our setting to show that the torsos of the Tutte decomposition are either 3-connected or cycles.

\end{example}

Let $G$ be a connected locally finite graph. The \emph{width} of a tree decomposition $(T,\cal V)$ of $G$ is $\sup \{|V_t|-1 : t\in V(T)\}$. The \emph{tree-width} $\tw (G)$ of $G$ is the infimum of the tree-width of its tree decompositions. Trees have tree-width 1. We say that a graph has \emph{bounded tree-width} if $\tw (G) < \infty$. The following is a straightforward proposition.

\begin{prop}
Let $G$ be a locally finite graph of bounded tree-width. Then $G$ is vertex-accessible.
\end{prop}

\subsection{Spherical graphs}

The \emph{topological representation} of a graph $G$ is the topological space obtained by representing every vertex as a point and every edge as an arc homeomorphic to $[0,1]$ whose end-points correspond to the end-vertices of the edge. We say that two graphs are \emph{homeomorphic} if and only if their topological representations are homeomorphic. Graph isomorphism implies graph homeomorphism but the converse is not true; for instance, any two cycles are homeomorphic.

 An \emph{embedding} of a graph $G$ in a topological space $U$ is a topological embedding of the topological representation of $G$ in $U$.  We denote such embeddings by $G \hookrightarrow U$ in a slight abuse of notation. A graph $G$ is \emph{planar} if every connected component of $G$ admits an embedding in $\RR^2$. A graph $G$ is \emph{spherical} if every connected component of $G$ admits an embedding in the 2-dimensional sphere $\SS$. The reader is referred to \cite{Gross} for further details on topological graph theory.

\begin{remark}
A connected countable graph is planar if and only if it is spherical. In this paper we will switch between the spherical and planar picture for convenience. Roughly speaking, on the one hand, the spherical framework is the natural setting to apply Whitney's and Imrich's Theorems below. On the other hand, the planar setting allows to consistently choose an inside and outside for each cycle of a graph. 
\end{remark}

Before recalling some further theory of spherical graphs, we show that torsos of the Tutte decomposition of a spherical 2-connected graph are spherical.

\begin{lemma}\label{lem:margullo}
Let $G$ be a finite 2-connected graph with Tutte decomposition $(T, \cal V)$, and $\alpha\colon G \hookrightarrow \SS$. Then, for each torso $G_t$ on a part $V_t \in \cal V$, with $t\in V(T)$, there exists {$\alpha_t\colon G_t \hookrightarrow \SS$} such that $\alpha_t|_{G[V_t]} = \alpha \circ i_{G[V_t]}$, where $i_{G[V_t]} : G[V_t] \rightarrow G$ is given by the identity map.
\end{lemma}

\begin{proof}
Define $\alpha_t|_{G[V_t]} = \alpha \circ i_{G[V_t]}$. We only need to find an $\alpha_t$-image for the virtual edges. Let $e$ be a virtual edge in $G_t$. The end-vertices of $e$ are the intersection of $G_t$ with some other torso $G_{t'}$. For 2-connectivity there must exist a part $V_{t'}\neq V_t$ of the Tutte decomposition and a path $P$ in $G[V_{t'}]$ such that $e \not\in E(P)$ and the end-vertices of $e$ and $P$ agree. Let $g_e$ be a homeomorphism between the topological representations of $e$ and $P$ and let $\alpha_t|_{e} = \alpha \circ g_e$. 

We claim that $\alpha_t$ is an embedding of $G_t$. Indeed, let $P$ and $Q$ be two paths defined as above associated respectively to two distinct virtual edges $e$ and $f$ of $G_t$. It suffices to show that $P$ and $Q$ are internally disjoint. The path $P$ is contained in some torso $G_{t_1}$ and $Q$ in $G_{t_2}$, with $t_1, t_2,$ and $ t$ pair-wise distinct. Moreover $t$ lies in the $T$-path between $t_1$ and $t_2$, so $V_{t_1} \cap V_{t_2} \subset V_t$. Hence, $P$ and $Q$ can only intersect at their end-vertices and the conclusion follows.
\end{proof}

The following is a classic characterization of spherical graphs.

\begin{theorem}[Kuratowski \cite{Kur}]
A finite graph $G$ is spherical if and only if it contains no subgraph homeomorphic to $K_{3,3}$ or $K_5$.
\end{theorem}

In \cite{Dir}, with an argument of Erd\H{o}s, Kuratowski's theorem is extended to infinite graphs with aid of the following lemma.

\begin{lemma}\label{prop:sphericalfin2inf}
A countable graph $G$ is spherical if and only if all its finite subgraphs are spherical.
\end{lemma}

A \emph{cyclic order} on a set $X$ is a set $\cal O \subset X^3$ which is \part Cyclic: $[x,y,z] \in \cal O$ implies $[z,x,y] \in \cal O$,

\part Asymmetric: $[x,y,z] \in \cal O$ implies $[x,z,y] \not \in \cal O$

\part Transitive: $[x,y,z], [x,z,w] \in \cal O$ implies $[x,y,w] \in \cal O$

\part Total: for all $x,y,z \in X$, either $[x,y,z] \in \cal O$ or $[x,z,y] \in \cal O$.

\trap  
Given a cyclic order $\cal O$ we let $\cal O^- := \{[x,y,z]\in X^3 : [z,y,x]\in \cal O$ denote the \emph{reversed} cyclic order.

A \emph{rotation} at a vertex $v \in V(G)$ of a graph $G$ is a cyclic order $\omega$ on $N_G(v)$. An embedding of $G \hookrightarrow\SS$ induces such a rotation on each vertex of $G$ once an orientation of $\SS$ is fixed. A \emph{spherical rotation system} on $G$ is a collection $\omega = \{\omega_v\}_{v\in V(G)}$ such that $\omega_v$ is a rotation at $v \in V(G)$ and every $\omega_v$ is induced by the same embedding $G \hookrightarrow \SS$. For a spherical rotation system $\omega$, the spherical rotation system obtained by reversing all rotations is denoted by $\omega^-$. Passing from $\omega$ to $\omega^-$ corresponds to reversing the orientation of $\SS$.

Let $\alpha,\beta\colon G \hookrightarrow \SS$ be two embeddings of a spherical graph $G$. We say that $\alpha$ and $\beta$ are \emph{topologically equivalent} if there exists a homeomorphism $h$ of $\SS$ such that $h\circ \alpha = \beta$, and that they are \emph{combinatorially equivalent} if they induce the same rotation system on $G$ up to rotations reversal. All embeddings of a bi-infinite path in $\SS$ are combinatorially equivalent. However, not all such embeddings are topologically equivalent: informally, the two ends of the bi-infinite path can be mapped to the same point of $\SS$ or to two different points.

\begin{theorem}[Whitney \cite{Whit}] All embeddings of a 3-connected spherical finite graph in $\SS$ are topologically equivalent.
\end{theorem}

Whitney's theorem, as stated, may not hold in the infinite case any longer. This is easy to see using an example along the lines of the bi-infinite path above. However, we still have the following.

\begin{theorem}[Imrich \cite{Imr}]
All embeddings of a 3-connected countable spherical graph in $\SS$ are combinatorially equivalent.
\end{theorem}

In the sequel we will use the following corollary of Whitney's theorem.

\begin{lemma}\label{lem:whitcore}
Let $G$ be a countable 3-connected locally finite spherical graph and let us fix $v \in V(G)$. Then, a rotation $\omega$ at the vertex $v$ is induced by a spherical embedding of $G$ if and only if it is induced by a spherical embedding of $B_r (v)$ for every $r \geq 1$.
\end{lemma}

\begin{proof}
The proof for necessity is trivial. To show sufficiency, we start by letting $\rho$ denote a fixed rotation system at $v$ which is induced by a spherical embedding of $G$.

For $G$ being 3-connected, there exists $R\in \NN$ such that there exist at least three internally disjoint paths in $B_R (v)$ between any two vertices in $B_1 (v)$. It follows that $B_1 (v)\subset H$, where $H$ is a 3-connected torso of the Tutte decomposition  of the block $[v]_{B_R (v)}$ of $v$ in $B_R (v)$. 

By hypothesis, we have embeddings $\alpha_R, \beta_R \colon B_R (v)\hookrightarrow \SS$, inducing $\om$ and $\rho$ respectively. Both embeddings can be restricted to $[v]_{B_R (v)}$, and induce embeddings $\tilde{\alpha}_R, \tilde{\beta}_R\colon H \hookrightarrow \SS$ by Lemma \ref{lem:margullo}. Whitney's theorem implies that either $\om = \rho$ or $\om = \rho^-$. Therefore, $\om$ is induced by a spherical embedding of $G$.
\end{proof}

Let $\FF_2$ denote the field on 2 elements. The \emph{edge space} $\cal E (G)$ of a graph $G$ is the vector space over $\FF_2$ consisting of finitely supported functions $E(G) \rightarrow \FF_2$. The \emph{cycle space} is the subspace of $\cal E (G)$ spanned by indicator functions of the form $\one_C$ where $C$ is a cycle of $G$. A \emph{2-basis} of a spherical graph $G$ is a family of cycles $\cal B$ which is 

\part simple: no edge of $G$ is contained in more than two elements of $\cal B$, and

\part generating: for every simple cycle $C$ of $G$ there exist $B_1, \dots, B_n$ such that $\one_C = \sum_{i=1}^n \one_{B_i}$ in the cycle space.
\trap

An embedding $\alpha\colon G \hookrightarrow \RR^2$  is \emph{accumulation-free} if every compact subset of $\RR^2$ contains at most finitely many vertices and edges of $G$. A cycle $C$ in $G$ is a \emph{facial cycle} of the embedding $\alpha$ if $\alpha(C)$ is the boundary of a connected component of $\RR^2\backslash \alpha(G)$.


\begin{theorem}[Thomassen \cite{Thom80}]\label{tma:Thom80}
Let $G$ be a 2-connected graph. Then, 
\part if $G$ has a 2-basis $\cal B$, then there exists an accumulation-free embedding $G \hookrightarrow \RR^2$ with $\cal B$ as its set of facial cycles, and

\part conversely, the facial cycles of an accumulation-free embedding of $G$ are a 2-basis.
\trap
\end{theorem}

The following is an example of a planar graph with no 2-basis.

\begin{example}\label{ex:miticu1}
For each $n \in \ZZ$, consider the set 
$$
V_n := \{(2^n,2^n), (2^n,-2^{n}), (-2^{n}, 2^n), (-2^{n}, -2^{n})\subset \RR^2.
$$
We define a graph $G$ by letting $V(G) = \bigcup_{n \in \ZZ} V_n$ and defining an edge between $(x,y)$ and $(x',y')$ if and only if both vertices agree in exactly one coordinate or $x/x' = y/y' \in \{1/2,2\}$. Let $\alpha$ denote the embedding of $G$ in $\RR^2$ obtained by drawing the edges of $G$ as straight segments. 

Let $C$ be a cycle of the form $C= G[V_n]$ for some $n \in \ZZ$. It is easy to see that $G$ is 3-connected, so by Imrich's Theorem every embedding of $G$ in $\RR^2$ is combinatorially equivalent in $\SS$ to $\alpha$. Since any embedding $\beta \colon G \hookrightarrow \RR^2$ is combinatorially equivalent to $\alpha$ in $\SS$, it is easy to see that both connected components of $\RR^2 \backslash \beta(C)$ will  contain infinitely many vertices of $G$. Hence, $G$ does not admit an accumulation-free embedding. By Thomassen's theorem, $G$ does not admit a 2-basis.
\end{example}


%

%

\subsection{Boolean ring of a graph}

Let $G$ be a conected locally finite graph. The \emph{coboundary} of a function $f \colon V(G) \rightarrow \FF_2$ is the set $\delta f := \{(u,v)\in E(G) : f(u)\neq f(v)\}$. The \emph{Boolean ring} of $G$, denoted by $\cal B G$ is the ring of functions $V(G) \rightarrow \FF_2$ with finite coboundary under the usual operations of pointwise addition and multiplication. The Boolean ring is equipped with an involution defined by $f^* := 1 + f$, where $1$ denotes the constant 1 function. For each $n\in \NN$ we let $\cal B_n G$ denote the subring of $\cal B G$ generated by those $f\in \cal B G$ with $|\delta f| \leq n$.

The reader is referred to \cite[Section II.2]{DD} for an introduction to the Boolean ring of a graph. Our aim in this section is to present some helpful facts about the Boolean ring.

\begin{lemma}
Let $G$ be a connected locally finite graph and $f \in \cal B G$. For every ray $R$ there exists a tail $R'$ of $R$ such that $f$ is constant on $V(R')$. Moreover, if $R,Q$ are equivalent rays of $G$, then $f$ agrees on tails of $R$ and $Q$. 
\end{lemma}

\begin{proof}
To prove the first clause, observe that, for $f \in \cal B G$, there are only finitely many edges of $R$ in $\delta f$. To prove the second clause, observe that for ray equivalence, there exist paths between tails of $R$ and $Q$ with edges outside the coboundary of $f$.
\end{proof}

It follows that a function $f \in \cal B G$ extends to a well-defined function $\tilde f : \ends (G) \rightarrow \FF_2$. Given two ends $\omega,\rho \in \ends(G)$, we say that an element $f \in \cal B G$ \emph{separates} the ends $\omega$ and $\rho$ if $\tilde f (\omega) \neq \tilde f(\rho)$.

\begin{prop}\label{prop:sep}
Let $G$ be a connected locally finite $n$-accessible graph for some $n\in \NN$. If $E$ is a generating set of $\cal B_n G$, then for any two ends $\omega,\rho \in \ends (G)$ there exists $f\in E$ separating $\omega$ and $\rho$.
\end{prop}

\begin{proof}
By the $n$-accessibility hypothesis, there exists $f \in \cal B_n G$ separating $\omega$ and $\rho$. Since $E$ generates $\cal B_n G$, there exist $f_1,\dots,f_m \in E$ such that $f \in \left<f_1, \dots, f_m \right>_{\cal B G}$. If no element in $f_1,\dots,f_m$ separates the two ends, then $f$ cannot separate the ends, what is a contradiction. 
\end{proof}

Let $E \subset \cal B G$. We say that $E$ is \emph{nested} if one of the following holds for any two $f,g\in E$: $$
f\geq g,\ g\geq f,\ f^* \geq g,\ \text{or}\ g^* \geq f.
$$
We say that $E$ is a \emph{tree set} if for any two $u,v\in V(G)$ there exist finitely many $f\in E$ such that $f(u) \neq f(v)$.

We define a map $\Psi : \Sep (G) \rightarrow \cal B G$ by letting $\Psi (S,B) \in \cal B G$ be the unique element in $\cal B G$ satisfying $\delta \Psi (S,B) = B$ and $\Psi (S,B) (S) = 1$.

\begin{lemma}
The map $\Psi$ satisfies the following properties. 
\part $\Psi$ is a well-defined bijection,
\part for $\cal S\subset \Sep (G)$ we have that $\cal S$ is nested if and only if $\Psi (\cal S)$ is nested, and 
\part if $\Psi (\cal S)$ is a tree set then $\cal S$ has property ($\ast$).
\trap
\end{lemma}

\begin{proof}
\part Indeed, the element $\Psi (S,B)$ is completely determined by definition. Let $u \in V(G)$. If $u \in S$ then $\Psi (S,B)(u) = 1$. If $u \in V(G) \backslash S$, then $u$ is connected by some path $P$ to $S$ with at most one edge in $B$. For $(S,B)$ being a separation, whether $P$ contains an edge in $B$ or not is independent of the choice of $P$. If $P$ contains an element in $B$, then $\Psi (S,B) (u) = 0$ and otherwise $\Psi (S,B) = 1$. 

Consider $f\in \cal B G$. Let $B := \delta f$ and $$
S:=\{u\in V(G) : f(u) = 1, \exists e\in B\ \text{such that}\ u\in e\}.
$$
Then $(S,B)$ is a separation and $\Psi (S,B) = f$. This identification determines the pre-image of $f$ uniquely, proving injectivity.

\part For a separation $(S,B)$, let $(A_1,A_2)$ denote its sides. Then $\Psi (S,B)$ is the element of $\cal B G$ determined by being the constant 1 function on $A_2$. The proposition now follows directly from the definitions.

\part Fix $(S,B)$ and $(S',B')$, and suppose that $(S'',B'')$ is such that $(S,B) < (S'',B'') < (S',B')$. If $S = S'$ then there are only finitely many possible $(S'',B'')$ by local finiteness of $G$. If $S\neq S'$, then $\Psi (S'',B'')$ takes different values on at least two distinct elements of $S$ and $S'$. By the tree set property, there can exist at most finitely many such $(S'',B'')$.

\trap
 
\end{proof}

\begin{example}\label{ex:notree}
There exist $\cal S$ sets of separations with property ($\ast$) such that $\Psi (\cal S)$ is not a tree set. For instance, in the graph $G$ of Example \ref{ex:miticu1}, consider vertices $u = (4,4)$ and $v= (1,1)$, together with the family of separations $(S_n,B_n)$ given as follows. For each $n\in \NN$, let $S_n$ be the set of vertices obtained as the union of the sets $\{(2,2),(2,-2),(-2,2)\}$, $\{(2^{-n},2^{-n}),(2^{-n},-2^{-n}),(-2^{-n},2^{-n})\}$, and the vertex sets of the shortest paths connecting $(2,-2)$ with $(2^{-n},-2^{-n})$ and $(-2,2)$ with $(-2^{-n},2^{-n})$. The sets $S_n$ are the vertex sets of cycles $C_n$ in $G$, and each of these $C_n$ defines an inside and outside by the embedding $\alpha$. We may choose $B_n$ to be, say, the outward-pointing edges for each of these $C_n$. 

Each of the $(S_n,B_n)$ separates the vertices $u$ and $v$, so $\{\Psi (S_n,B_n)\}_{n\in \NN}$ is not a tree set. However, the separation order on the $(S_n,B_n)$ gives an $\NN$-order with finite chains, so $\{(S_n,B_n)\}_{n\in \NN}$ has property ($\ast$).
\end{example}

An element $f \in \cal B_n G$ is $n$-thin if $f\not\in \cal B_{n-1} G$ and $|\delta f| = n$. An element $f\in \cal B G$ is \emph{thin} if it is thin for some $n\in \NN$. Thin elements are \emph{connected}, where $f\in \cal B G$ is said to be connected if the induced graph $G[\{u\in V(G) : f(u) = 1\}]$ is connected.

\begin{prop}\label{prop:locfin}
For every $n\in \NN$ and $u,v \in V(G)$ there exist finitely many thin elements $f\in \cal B_n G$ separating $u$ and $v$.
\end{prop}

\begin{proof}
The following argument follows along the lines of that in \cite[Proposition 4.1]{Thom80}. First, we show that for every edge $e \in E(G)$ there exist at most finitely many $f\in \cal B_n G$ with $e\in \delta f$. The proof follows by induction. The case $n=1$ is trivial. For $n > 1$, suppose that $f\in \cal B_n G$ is thin and satisfies that $e\in \delta f$. For thinness and $n > 1$, there must exist a path $P$ between the end-vertices of $e$ not containing $e$. We have that $f \in \cal B_{n-1} (G\backslash \{e\})$ is still thin, and that $\delta f$ must contain some edge of $P$. Applying the induction hypothesis, there exist finitely many such possible $f \in \cal B_{n-1} (G\backslash \{e\})$ with $\delta f$ containing some edge of $P$. Therefore, there exist finitely many $f\in \cal B_n G$ for which $e$ is contained in $\delta f$. 

The proof is concluded by a similar argument. Fix an arbitrary path $P$ with end-vertices $u$ and $v$. If $f\in \cal B_n G$ separates $u$ and $v$, then there exists $e\in E(P)$ such that $e\in \delta f$. Therefore, there can exist only finitely many possible $f \in \cal B G$ separating $u$ and $v$.
\end{proof}

\subsection{Borel graphs and equivalence relations}

A \emph{standard Borel space} is a pair $(X,\cal B)$ where $X$ is a set and $\cal B$ a $\sigma$-algebra on $X$ generated by a Polish topology on $X$. The standard Borel space of all finite subsets of a standard Borel space $X$ with cardinality at most $k\in \NN$ is denoted by $[X]^{\leq k}$, and $[X]^{<\infty}$ denotes the union $\bigcup_{k\in \NN } [X]^{\leq k}$. A \emph{Borel probability measure} $\mu$ is a measure on $(X, \cal B)$ such that $\mu (X) = 1$. Whenever it is clear from the context, we avoid explicit reference to the $\sigma$-algebra. 

Let $f\colon X \rightarrow Y$ be a function between sets $X,Y$. Recall that a \emph{section} of $f$ is a function $g\colon Y \rightarrow X$ such that $f\circ g$ is the identity on the range of $f$. The following theorem compiles some useful facts about functions between standard Borel spaces. For details, the reader is referred to \cite{kech}.

\begin{theorem} Let $X,Y$ be standard Borel spaces and $f\colon X\rightarrow Y$ a function. Then $f$ is Borel if and only if $\graph (f)$ is Borel. Furthemore, if $f$ is Borel, then the following  propositions hold.

\part (Lusin-Souslin) If $A \subset X$ is Borel and $f|_A$ is injective, then $f(A)$ is Borel. Moreover, if $A\subset X$ is Borel and $f|_A$ is countable-to-one, then $f(A)$ is Borel.

\part (Lusin-Novikov) If $f$ is countable-to-one, then it admits a Borel section.

\trap
\end{theorem}

An equivalence relation $R$ on a standard Borel space $X$ is Borel if it is a Borel subset $R \subset X \times X$. We say that a Borel equivalence relation $R$ is \emph{countable} (resp. \emph{finite}) if its equivalence classes are countable (resp. finite).   The $R$-equivalence class of $x\in X$ is denoted by $[x]_R$. The \emph{full group} $[R]$ of a countable Borel equivalence relation $R$ on a standard Borel space $X$ is the Borel set of those Borel bijections $f\colon X\rightarrow X$ such that $\graph (f) \subset R$. We say that a Borel probability measure $\mu$ is $R$\emph{-invariant} if $f_* \mu = \mu$ for every $f \in [R]$. A \emph{p.m.p.~countable Borel equivalence relation} is a pair $(R,\mu)$ where $R$ is a countable Borel equivalence relation and $\mu$ is an $R$-invariant Borel probability measure on the vertex space of $R$. A \emph{graphing} of such $(R,\mu)$ is a Borel graph $\cal G$ such that $\Rel (\cal G) = R$ a.s.

The set of equivalence classes of $R$, denoted by $X/R$, admits a \emph{quotient Borel structure} where $A \subset X/R$ is Borel if and only if $\bigcup A \subset X$ is Borel. We say that $R$ is \emph{smooth} if $X/R$ is a standard Borel space with its quotient Borel structure. A \emph{complete section} for $R$ is a Borel subset $A\subset X$ intersecting every orbit. A \emph{Borel transversal} for $R$ is a complete section $A \subset X$ intersecting each orbit of $R$ in exactly one point. A \emph{Borel selector} of $R$ is a Borel map $f:X\rightarrow X$ such that $\graph(f)\subset R$ and the image of $f$ is a Borel transversal. A Borel equivalence relation admits a Borel selector if and only if it admits a Borel transversal. If $R$ has a Borel transversal, then it is smooth. Finite Borel equivalence relations admit Borel transversals. The reader is referred to \cite{KM} for further details on Borel equivalence relations.

An example of smooth Borel equivalence relation $R$ is defined on $\bigcup_{n\in \NN} X^n$ by letting $x \sim_R y$, for $x\in X^n$ and $y \in X^m$, if and only if $m = n$ and there is $k \in \{0, 1,\dots,n-1\}$ such that, for every $l \in \{1, \dots, n\}$, the equality $x_{l - k} = y_l$ holds, where sum in the indices is modulo $n$. The equivalence relation $R$ is Borel and finite, so $\bigcup_{n\in \NN} X^n/R$ is a standard Borel space with its quotient Borel structure. We denote this space by $\Cyclic (X)$, as it can be easily identified with the set of finite cyclic orders on the elements of $X$.

A \emph{Borel graph} $\cal G$ on a standard Borel space $X$ is a symmetric Borel subset of $X \times X$. We assume that $\cal G$ has countable connected components. We let $\proj (\cal G)$ denote the Borel subset of $X$ consisting of those vertices with an incident $\cal G$-edge. Indeed such set is Borel by Lusin-Souslin's Theorem. The set $X$ is referred to as the \emph{vertex space} of $\cal G$. The $\cal G$-connected component of a vertex $x\in X$ is denoted by $\cal G_x$. An \emph{extension} of a Borel graph $\cal G$ is a Borel graph homomorphism $f\colon\cal H \rightarrow \cal G$, where $\cal H$ is a Borel graph, such that $f$ restricted to each component of $\cal H$ is a graph isomorphism onto its image. We also say that such $\cal H$ is an extension of $\cal G$.

  A Borel graph $\cal G$ induces a countable Borel equivalence relation $\Rel (\cal G)$ on its vertex space $X$ defined by letting $x \sim y$, for $x,y \in X$, if and only if $x$ and $y$ are in the same component of $\cal G$. We say that $\Rel (\cal G)$ is \emph{Borel treeable} if there exists an acyclic Borel graph $\cal T$ on $X$ such that $ \Rel (\cal G) = \Rel (\cal T)$. If $\mu$ is a Borel probability measure on $X$, we say that $\Rel (\cal G)$ is $\mu$\emph{-treeable} if there exists a Borel set $X_0 \subset X$ with $\mu (X_0) = 1$ such that the restriction $\Rel (\cal G)|_{X_0}$ is Borel treeable. Finally, $\Rel (\cal G)$ is \emph{measure treeable} if it is $\mu$-treeable for every Borel probability $\mu$ on $X$.
  
\begin{theorem}[Gaboriau \cite{Gab}]\label{tma gabo}
Let $\cal G$ be a Borel graph on a standard Borel space $X$ and let $A \subset X$ be a Borel set intersecting every connected component. Then $\Rel(\cal G)$ is treeable if and only if $\Rel (\cal G)|_A$ is treeable.
\end{theorem}

A countable Borel equivalence relation $R$ is \emph{hyperfinite} if there exists an increasing sequence of finite Borel subequivalence relations $(R_n)_n$ such that $R = \bigcup_{n\in \NN} R_n$. More generally, a p.m.p.~countable Borel equivalence relation $(R,\mu)$ is \emph{approximable} if there exists an increasing sequence of Borel subequivalence relations $(R_n)_{n\in \NN}$ such that $R = \bigcup_{n\in \NN} R_n$ and, for any Borel subset $A$ of the vertex space, we have that $R_n|_A = R$ for some $n$ implies $\mu (A) = 0$. 

A \emph{p.m.p.~extension} of p.m.p.~countable Borel equivalence relation $(R,\mu)$ is another p.m.p.~countable Borel equivalence relation $(Q,\nu)$ together with a Borel map $f\colon Y \rightarrow X$, where $Y$ and $X$ are the vertex spaces of $Q$ and $R$ respectively, such that $f_* \nu = \mu$ and $y \sim_Q y'$ implies $f(y) \sim_R f(y')$ for almost every $y,y'\in Y$.

 A p.m.p.~countable Borel equivalence relation $(R,\mu)$ has \emph{measured property (T)} if and only if no extension of $(R,\mu)$ is approximable. This definition is equivalent to that introduced by Moore \cite{Moo} and Zimmer \cite{Zim}. This follows from soon available work of the author with \L{}ukasz Grabowski and Samuel Mellick. The fact that p.m.p.~countable Borel equivalence relations are not approximable is proven in Pichot's thesis \cite{pichot}.

\subsection{Aperiodic Borel forests}

The aim of this subsection is to prove the following proposition, which will be helpful in Section 4.

\begin{prop}\label{prop:forex}
Let $\cal G$ be a locally finite Borel graph. Then there exists an aperiodic Borel subforest $\cal F \subset \cal G$ intersecting every $\cal G$-component.
\end{prop}

The proof of the proposition relies on two lemmas.

\begin{lemma}\label{lem:forex1}
Let $\cal G$ be a locally finite Borel graph. Then there exists $\cal T \subset \cal G$ Borel acyclic graph such that all components of $\cal T$ contain at least 2 vertices.
\end{lemma}

\begin{proof}
Let $\cal G = \bigcup_{n\in \NN} \cal G_n$ be defined by Feldman-Moore's Theorem, so each $\cal G_n$ is a matching. We may assume without loss of generality that for every $x \in X \backslash \proj (\cal G_1)$ we have that $N_{\cal G} (x) \subset \proj (\cal G_1)$. 

Let $\cal H$ be a Borel subset of $\cal G$ consisting of exactly one edge $h_x \in \cal G$ for each $x \in X \backslash \proj (\cal G_1)$ such that one end-vertex of $h_x$ is $x$. By maximality the other end-vertex of $h_x$ is in $\proj (\cal G_1)$. Such $\cal H$ may be constructed by letting $h_x$ be the edge incident to $x$ with lowest label in the Feldman-Moore decomposition of $\cal G$.

We claim that $\cal T := \cal G_1 \cup \cal H$ has the desired properties. Let us first show that $\cal T$ is acyclic. Observe that all vertices in $X\backslash \proj (\cal G_1)$ have degree 1 in $\cal T$. Therefore, a cycle in $\cal T$ would be contained in $\cal G_1$. But this is not possible for $\cal G_1$ being a matching. We deduce that $\cal T$ is acyclic.

Now let $x \in X$. If $x\in \proj (\cal G_1)$ then $x$ is connected by a $\cal G_1$-edge to some other vertex, and so $|[x]_{\cal T}| \geq 2$. When $x \in X \backslash \proj (\cal G_1)$, then there exists some $h_x \in \cal H$, and the conclusion follows.
\end{proof}

\begin{lemma}\label{lem:forex2}
Let $\cal G$ be a locally finite Borel graph and $\cal T \subset \cal G$ a Borel subforest with finite components of size at least $n\in \NN$. Then, there exists a Borel subforest $\cal T \subset \cal T' \subset \cal G$ with all components of size at least $2n$.
\end{lemma}

\begin{proof}
The Borel equivalence relation $\Rel (\cal T)$ is finite by hypothesis, so its quotient Borel space $X/\cal T$ is a standard Borel space. Let $\pi \colon X \rightarrow X/\cal T$ be the associated quotient map. Let $\cal H := \pi\times \pi (\cal G)$ be the Borel graph on $X/\cal T$ obtained as a contraction of $\cal T$ components and then identifying multiple edges. By local finiteness of $\cal G$ and finiteness of $\cal T$ components we have that $\cal H$ is locally finite. 

Let $\cal Q \subset \cal H$ be obtained by Lemma \ref{lem:forex1}, and let $\cal T'$ be the union of $\cal T$ with exactly one edge $h_e \in \cal G$ such that $\pi\times \pi (h_e) = e$ for each $e \in \cal Q$. By construction of $\cal Q$, every $\cal T$-component is connected to some other $\cal T$-component by some $\cal T'$-edge, so the conclusion on $\cal T'$-component sizes follows. Acyclicity follows from the fact that $\cal Q$ is a deformation retract of $\cal T'$.
\end{proof}

\begin{proof}[Proof of Proposition \ref{prop:forex}]

Let $\cal T_1\subset \cal G$ be defined by Lemma \ref{lem:forex1}. Given $n\in \NN$, let $X_n \subset X$ be the Borel $\cal G$-invariant subset obtained as the union of those $\cal G$-components where every $\cal T_n$-component is finite. Observe that we have already found a solution to our problem in $X \backslash X_n$. Then, we let $\cal T_{n+1}$ be obtained by Lemma \ref{lem:forex2} for $\cal T_n |_{X_n}$ and $\cal G|_{X_n}$. Let $\cal F := \bigcup_{n\in \NN} \cal T_n$.

We have that  $X_n \supset X_{n+1}$ for every $n\in \NN$. If $\bigcap_{n\in \NN} X_n = \emptyset$, then every $x \in X$ is in some $X_m$ and so $\cal T_m$ already has an infinite component in $[x]_{\Rel (\cal G)}$. Otherwise, for every $x\in \bigcap_{n\in \NN} X_n$ we have that $[x]_{\cal T_n} \geq 2n$ for every $n\in \NN$, so $|[x]_{\cal F}| = \infty$.

\end{proof}

\section{Tree decompositions in Borel graphs}
\label{sec borel tree}

Theorem \ref{tma:maintech} is the main theorem of this section. There, roughly speaking, we reduce the question of treeability of $\Rel (\cal G)$, for a given Borel graph $\cal G$, to treeability of an auxiliar equivalence relation $Q$, called the \emph{levels equivalence relation}. Such equivalence relation $Q$ is defined on an auxiliar standard Borel space by two vertices being in the same part of a previously introduced tree decomposition. Most of this chapter is devoted to the construction of such tree decompositions in a suitable way for the Borel setting, adapting the tools introduced in the preliminaries. The rest of the paper is devoted to applications of Theorem \ref{tma:maintech}. We point out that there is a very close relation between Theorem \ref{tma:maintech} and results in \cite{Tse}. This is discussed in Remark \ref{rem:tse} at the end of the section.

Let $\cal G$ be a locally finite Borel graph on a standard Borel space $X$. We let $\Sep (\cal G) \subset [X]^{< \infty} \times [\cal G]^{< \infty}$ denote the Borel subset of those $(S,B)\in [X]^{< \infty} \times [\cal G]^{< \infty}$ such that all vertices in $S$ belong to the same $\cal G$-component and $(S,B)$ is a separation in such component. A \emph{Borel separation system} is a Borel subset $\cal S \subset \Sep (\cal G)$ such that for every component $G$ of $\cal G$ the family $\cal S_G:= \{(S,B)\in \cal S : S\subset V(G)\}$ is a separation system.

If $(S,B) \in \cal S$, then $\Sides(S,B)$ denotes the ordered pair of sides of $(S,B)$ in its connected component $G_S$, and $\Sides (\cal S) := \{\Sides (S,B) \colon (S,B)\in \cal S\}$. We say that $\cal S$ is \emph{nested} if for every connected component $G$ of $\cal G$, the family $\cal S_G$ is nested. Similarly, $\cal S$ has \emph{property} ($\ast$) if, for every connected component $G$ of $\cal G$, the separation system $\cal S_G$ has property ($\ast$). We define the pair $(\cal T_{\cal S},\cal V_{\cal S})$ obtained from a nested Borel separation system $\cal S$ by $\cal T_{\cal S} = \bigcup \cal T_{\cal S_G}$ and $ \cal V_{\cal S} = \bigcup \cal V_{\cal S_G}$ where $G$ runs over the connected components of $\cal G$ and $(\cal T_{\cal S_G},\cal V_{\cal S_G})$ is defined as in the preamble of Proposition \ref{prop:CHM}.

\begin{prop}\label{prop:Z}
Let $\cal G$ be a locally finite Borel graph on a standard Borel space $X$ and $\cal S$ a nested Borel separation system. Then, the set $Z = \{(x,t)\in X\times V(\cal T_{\cal S}) : x\in V_t\}$ admits a standard Borel space structure such that the map $p\colon (x,t) \in Z \mapsto x\in X$ is Borel and countable-to-one.
\end{prop}

\begin{proof}
First, we endow $Z$ with its standard Borel space structure and then prove (a) and (b). For each $x \in X$, let $\cal S_x \subset \cal S$ denote the set of pairs $(S,B) \in  \cal S$ with $x\in S$. We consider the Borel subset  \begin{equation*}
Y \colon= \{(x,(S,B))\in X \times \cal S \colon (S,B) \in \cal S_x\} \cup U \subset X\times \cal S
\end{equation*}
and define an equivalence relation $R$ on $Y$ by letting $(x,(S,B)) \sim_R (y,(S',B'))$ if and only if $x= y$ and $(S,B) \sim_{\cal S_{\cal G_x}} (S',B')$ or $(S,B) = (S',B')$. The equivalence relation $\sim_{\cal S_{\cal G_x}}$ is defined as in the preamble of Proposition \ref{prop:CHM}.

\begin{claim}\label{prop:suave}
$R$ is a smooth Borel equivalence relation.
\end{claim}

\begin{proof}[Proof of claim]
$R$ is Borel since the order $\leq$ on separations determining the equivalence relations $\sim_{\cal S_{\cal G_x}}$ is determined in finite neighbourhoods of $\cal G$. Indeed, let $(A_1,A_2)$ and $(A_1',A_2')$ be the sides of two separations $(S,B)$ and $(S',B')$ respectively. In order to check that $A_1 \subset A_1'$ it suffices to show that both $S$ and $B$ are contained in $A_1'$, and such containment is witnessed in a $\cal G$-neighbourhood of finite radius around $S'$. A similar argument holds for the inequality $A_2' \subset A_2$ required to complete the requirements for $(S,B) \leq (S',B')$.

The map $q\colon(x,(S,B))\in Y \mapsto x \in X$ is Borel and countable-to-one, since $x\in S$, the set $S$ is finite, and components of $\cal G$ are countable. By a standard application of Lusin-Novikov's Theorem \cite[Exercise 18.15]{kech}), there exist partial Borel functions $f_i \colon X_i \subset X \rightarrow Y$ with $i \in \NN$ such that $q^{-1} (x) = \{f_i (x) \colon i \in \NN, x \in X_i\}$ and $|q^{-1} (x)|= |\{i \in \NN \colon x \in X_i\}|$. For each $i \in \NN$, let $A_i \subset X_i$ be the Borel set of those $x \in X_i$ such that $f_i (x) \not\sim_R f_j (x)$ for any $j< i$. The set $\bigcup_{i\in \NN} f_i (A_i) \subset Y$ is Borel by Lusin-Souslin's Theorem and a transversal of $R$ by construction.
\end{proof}

Let $U \subset X$ be the Borel subset of those $x\in X$ for which there exists no $(S,B)\in \cal S$ with $x\in S$ but there exists $t\in V(\cal T_{\cal S})$ such that $x\in V_t$. Equivalently, there exists $(S,B) \in \cal S$ such that $x$ lies in the part of $\cal T_{\cal S_{\cal G_x}}$ defined by $(S,B)$, so there is no other $(S',B')\in \cal S$ separating $x$ from $(S,B)$. Arguing similarly as in the claim it follows that $U$ is indeed Borel, as the latter is a Borel condition. We observe, merely as a warning, that even though $q (Y) \cup U$ will intersect every $\cal G$-component, it may be a strict Borel subset of $X$ in cases like that of Example \ref{ex:notree}.

It follows from the claim and Proposition 6.3 in \cite{KM} that $Z' \colon= (Y/R) \cup U$ is a standard Borel space. The standard Borel space structure on $Z$ is induced by a bijection between $Z$ and $Z'$. This bijection is established by associating to each equivalence class $[x, (S,B)]_R \in Z'$ the pair $(x,t) \in Z$ where $t = [(S,B)]_{\cal S_{\cal G_x}}$ when $x\not\in U$, or, in case $x \in U$, the only such $t$ for which $x\in V_t$. 

To prove (a), the fact that $p$ is a Borel map follows from the map $q$ above being Borel as well as the quotient map $Y \rightarrow Y/R$. Since the map $q$ in the claim is countable-to-one, and the pre-image of a vertex $x \in X$ by $p$ consists of $R$-equivalence classes of $q^{-1} (x)$, then $p^{-1}(x)$ is countable.

\end{proof}
 
 We let $\tilde R$ be the countable Borel equivalence relation on $Z$ defined by $(x,t) \sim_{\tilde R} (y,t')$ if and only if $x \sim_{\Rel (\cal G)} y$ and $Q$ be the countable Borel equivalence relation on $Z$ defined by $(x,t) \sim_Q (y,t')$ if and only if $t = t'$. We will refer to $Q$ as the \emph{levels equivalence relation (associated to $\cal S$)}. We refer to an equivalence class of $Q$ as a \emph{level}
 
Theorem \ref{tma:maintech} below will allow us to deduce treeability of $R$ in terms of treeability of $Q$. In order to apply Theorem \ref{tma:maintech} in the following sections, we will need to graph $Q$. However, this is not needed to prove Theorem \ref{tma:maintech}, so for now we will restrict this discussion to an illustrative example.

\begin{example}[Tutte decomposition]\label{ex:Tuttemes}
Let $\cal G$ be a Borel graph with 2-connected components on the standard Borel space $X$. Let $\Sep_2 (\cal G) \subset \Sep (\cal G)$ be the Borel set of separations $(S,B)$ such that $|S| = 2$. In the proof of Claim \ref{prop:suave}, we remarked that the order $\leq$ on separations is determined locally. Hence, the subset $\cal S \subset \Sep_2 (\cal G)$ consisting of those separations $(S,B) \in \Sep_2 (\cal G)$ which are nested with every other separation of $\Sep_2 (\cal G)$ in their component is Borel. By Proposition \ref{prop:estar}, the separation system $\cal S'$ induced by $\cal S$ is a Borel separation system with property ($\ast$).

One graph we can consider on $Z$ is the Borel graph $\cal L$ given by $(x,t)\sim (y,t')$ if and only if $x\sim_{\cal G} y$ and $t = t'$. This Borel graph satisfies that $\Rel (\cal L) \subset Q$, but the inequality could be strict. Components of $\cal L$ are subgraphs of $\cal G$. Additionally, consider $\cal A$ be the set of edges on $Z$ defined by $(x,t)\sim (y,t')$ if and only if $t = t'$ and there exists $t''\in V(\cal T_{\cal S})$ such that $x,y\in V_{t''}$. The latter condition is equivalent to the pairs $(x,t'')$ and $(y,t'')$ defining elements of $Z$. Then, the components of the Borel graph $\cal L\cup \cal A$ are the torsos of the Tutte decompositions of $\cal G$-components, so we will have that $\Rel (\cal L\cup \cal A) = Q$ and that $\cal L\cup \cal A$-components are either 3-connected or cycles.
\end{example}

The following is the main theorem of this section.



\begin{theorem}\label{tma:maintech}
Let $\cal G$ be a locally finite Borel graph on a standard Borel space $X$, let $\cal S$ be a nested Borel separation system with property ($\ast$), and let $\mu$ be a Borel probability measure on $X$. If $Q$ is Borel treeable (resp.~measure-treeable), then $\Rel (\cal G)$ is Borel treeable (resp.~measure-treeable).
\end{theorem}

\begin{proof}

Let us first deal with the Borel treeability case. Let $\cal Q$ be a Borel treeing of $Q$ and $\cal W \subset Z\times Z$ the Borel set of edges defined by $(x,t) \sim_{\cal W} (x,t')$ if and only if there is a $\cal T_{\cal S}$-edge between $t$ and $t'$. The graph $\cal W$ is Borel since $(x,t) \sim_{\cal W} (x,t')$ if and only if there exists $(S,B) \in \cal S$ such that $(S,B) \in t$ and $(S,B)^- \in t'$. 

Our next goal is to drop enough $\cal W$ edges from $\cal Q \cup \cal W$ as to obtain an acyclic graph inducing $Q$. To this end, let $\cal J$ be the Borel equivalence relation on $\cal W$ defined by pairs of the form $((u,t),(u,t')) \sim_{\cal W} ((v,t),(v,t'))$. The equivalence relation $\cal J$ is finite, as for each pair $(t,t') \in E(\cal T_{\cal S})$ there exists a unique $(S,B)\in \cal S$ such that $t = [(S,B)]_{\cal S}$ and $t' = [(S,B)^-]_{\cal S}$, so if $((u,t),(u,t')) \in \cal W$, then $u\in S$. Therefore, we may let $\cal W'$ be a Borel transversal for $\cal W$.

We now show that $\cal Q \cup \cal W'$ is acyclic and that $\Rel (\cal Q \cup \cal W') = \tilde R$. In order to do this, we show that given any two $\tilde R$-related $(x,t),(y,t') \in Z$, there exists a unique path $\tilde P \subset \cal Q \cup \cal W'$ with end-vertices $(x,t)$ and $(y,t')$. By a standard induction argument and Proposition \ref{prop:CHM}, we may assume without loss of generality that $(t,t')\in E(\cal T_{\cal S})$. Then, the unique path from $t$ to $t'$ in $\cal T_{\cal S}$ is determined by a unique $(S,B)\in \cal S$ such that $t = [(S,B)]_{\cal S}$ and $t' = [(S,B)^-]_{\cal S}$. Letting $u\in S$ we have that $((u,t),(u,t')) \in \cal W$, so there exists a unique edge  of the form $((v,t),(v,t'))\in \cal W'$. Then, the path $\tilde P$ is obtained as the concatenation of the unique $\cal Q$-path from $(x,t)$ to $(v,t)$, with the edge  $((v,t),(v,t'))\in \cal W'$, and then the unique $\cal Q$-path between $(u,t')$ and $(y,t')$. This implies that $\tilde R$ is Borel treeable.

Since $p\colon(x,t) \in Z \mapsto x \in X$ is countable-to-one by Proposition \ref{prop:Z}.(a), it admits a Borel section $f$ by Lusin-Novikov's Theorem. Let $Z_0\colon= f(X)$, and note that $Z_0$ is a  Borel subset of $Z$ by Lusin-Souslin's Theorem. By Theorem \ref{tma gabo}, there exists an acyclic Borel graph $\cal T$ on $Z_0$ such that $\Rel (\cal T) = \tilde R|_{Z_0}$. Again by Lusin-Souslin's Theorem, the graph $p (\cal T)$ is Borel. By definition of $\tilde R$ we have that $\Rel (p(\cal T)) = \Rel(\cal G)|_{p(f(X))}$. Finally, since $p$ is injective on $Z_0$, the Borel graph $p(\cal T)$ is acyclic. It follows from Theorem \ref{tma gabo} that $\Rel (\cal G)$ is Borel treeable.

For the case of measure-treeability, let $\mu$ be a Borel probability measure on $X$, and let $X_n \subset X$ be the Borel subset of those $x\in X$ such that $|p^{-1} (x)| = n$, for $n\in \NN \cup \{\infty\}$. Applying the Lusin-Novikov theorem again as in the proof of the claim, we obtain, for each $n \in \NN \cup \{\infty\}$, Borel functions $f_i^{(n)}\colon X_n \rightarrow Z$ with pairwise disjoint graphs such that $p^{-1} (x) = \{f_i^{(n)} (x)\}_{i=1}^n$ for each $x\in X_n$ and $n \in \NN \cup \{\infty\}$. We let $\nu$ be any Borel probability measure in the class of $\sum_{n \in \NN\cup\{\infty\}} \sum_{i=1}^n (f_i^{(n)})_* \mu$. Such $\nu$ exists for the latter being a $\sigma$-finite Borel measure on a standard Borel space.

Upon discarding a $\nu$-null set, we may assume that $\nu$ is $\tilde R$-quasi-invariant, since $\tilde R$ is countable. Let $N\subset Z$ be a $\nu$-null Borel $\cal Q$-invariant subset such that $Q|_{Z\backslash N}$ is Borel treeable. Let $Z'$ denote the complement of the $\tilde R$-saturation of $N$. Then $Z'$ is $\tilde R$-invariant and $Q|_{Z'}$ is Borel treeable.  

Moreover, for $p$ being countable-to-one $p (Z')$ is Borel by Lusin-Souslin's Theorem. By definition of $\tilde R$, we have that $p(Z')$ is $\Rel (\cal G)$-invariant and, by construction of $\mu$, that $p(Z')$ has full $\mu$-measure. The proof is concluded running the same argument as above, substituting $Z'$ for $Z$, and letting $\cal Q$ be a Borel treeing of $Q|_{Z'}$.

\end{proof}

\begin{remark}\label{rem:tse}
The above theorem is just a slight generalization of the treeability condition  \cite[Corollary 4.2]{Tse}. Our theorem includes the cases of $\mu$-treeability and measure treeability. Also, it does not require a condition on the number of ends of $\cal G$-components. The framework of the auxiliary space $Z$ will also be helpful in the following sectionswhen we need to graph $Q$. Apart from this slight differences, the construction here presented follows the common lines of using tree decompositions  and yields similar results as in \cite{Tse}.
\end{remark}

\section{Accessibility and subgraph-closed families}

The next two sections are devoted to provide examples where Theorem \ref{tma:maintech} can be applied. There are two concrete families of Borel graphs which are proven to be treeable in the sequel. The first of them is that of uniformly locally finite Borel graphs with components of bounded tree-width, which are proven to be Borel treeable in Corollary \ref{cor:bddtw}. The second is that of Borel graphs with accessible planar components. These latter graphs are proven to be measure treeable in the next section.

The two examples of the above paragraph will follow from a more general corollary of Theorem \ref{tma:maintech}, namely Theorem \ref{tma:Faccessible}. The statement of this result involves the following generalization of accessibility to the setting of Borel graphs. Given a class $\cal F$ of locally finite connected graphs, let $\Bor (\cal F)$ denote the class of Borel graphs with components in $\cal F$. The introduction of $\Bor (\cal F)$ will just serve as a shortcut to write $\cal G \in \Bor (\cal F)$ instead of ``$\cal G$ has components in $\cal F$''.

\begin{definition}\label{def:boracc}
Let $\cal F$ be a family of locally finite connected graphs. We say that a locally finite Borel graph $\cal G$ is \emph{Borel $\cal F$-accessible} if there exists a nested Borel separation system $\cal S \subset \Sep (\cal G)$ with property $(\ast)$ and a Borel graph $\cal H $ with $\Rel (\cal H) = Q$, where $Q$ is the levels equivalence relation associated to $\cal S$, such that $\cal H \in \Bor (\cal F)$.
\end{definition}

As discussed above, the motivation to study Borel $\cal F$-accessibility comes from the following straightforward corollary of Theorem \ref{tma:maintech}. 

\begin{theorem}\label{tma:Faccessible}
Let $\cal F$ be a family of locally finite connected graphs. If every  Borel graph in $\Bor (\cal F)$ is Borel (resp.~measure) treeable, then so is any Borel $\cal F$-accessible locally finite graph. 
\end{theorem}

Let $\cal F_1 \subset \cal F$ denote the subfamily of one-ended graphs in $\cal F$. The following theorem justifies the use of the term accessibility in the definition of Borel $\cal F$-accessibility.

\begin{theorem}\label{tma:acces}
Let $\cal F$ be a minor-closed family of locally finite connected graphs. Then, Borel graphs in $\Bor(\cal F)$ with accessible components are Borel $\cal F_1$-accessible. 
\end{theorem}

\begin{proof}

Let $\cal G \in\Bor (\cal F)$ with vertex space $X$. First, we show that there exists a Borel $\cal G$-invariant map $\kappa \colon X \rightarrow \NN$ such that $\kappa (x)$ is the least natural number such that $\cal G_x$ is $\kappa (x)$-accessible. To construct such $\kappa$, let each vertex $x\in X$ enumerate each separation in its component using Lusin-Novikov's theorem. Fix an $x\in X$ and let $((S_n,B_n))_{n\in \NN}$ be its corresponding enumeration. For each $n\in \NN$, let $k_n \in \NN$ be the least natural number such that for any two ends $\omega,\omega'\in \ends (\cal G)$ separated by $(S_n,B_n)$ there exists $(S,B) \in \Sep (\cal G)$ with $|B|\leq k_n$ separating $\omega$ and $\omega'$. In order to determine $k_n$ in a Borel fashion one may, more precisely, endow the set of bi-infinite paths with tails separated by $(S_n,B_n)$ a structure of standard Borel space using an inverse limit and then again use the enumeration $((S_n,B_n))_{n\in \NN}$. Such $k_n$ is finite by the accessibility hypothesis. Finally, define $\kappa (x) := \sup \{k_n\}_{n\in \NN}$.  Again, by the accessibility hypothesis, we have that $\kappa (x) \in \NN$. 

Our next goal is to show that there exists a nested Borel separation system $\cal S$ with property ($\ast$) such that for any two ends of $\cal G$ in the same component there exists a separation in $\cal S$ separating the two ends. In order to prove this, we run the Dicks-Dunwoody argument \cite[Section II.2]{DD} in the Borel setting.

We let $\cal S_1 \subset \Sep (\cal G)$ be the Borel subset of those $(S,B) \in \Sep (\cal G)$ with $|B|=1$. It is easy to see that $\cal S^1$ is nested, and that $\Psi (\cal S_1)$ is thin. Let us suppose that we are given $\cal S^n$ for some $n< \kappa$ such in each $\cal G$-component $G$ the Boolean ring elements of $\Psi (\cal S^n_G) \subset \cal B G$ consist of thin elements generating $\cal B_n  G$. This is satisfied for $n=1$. Our aim is to find $\cal S^{n+1}$ such that $\cal S^n \subset \cal S^{n+1}$, and $\Psi (\cal S^{n+1}_G)$ is a thin generating set of $\cal B_{n+1} G$ for every $\cal G$-component $G$. By the Dicks-Dunwoody argument \cite[Section II.2]{DD}, it suffices for $\cal S^{n+1}$ to be the separation system induced by a maximal nested set of separations containing $\cal S^n$ and such that $\Psi (S,B)$ is thin for every $(S,B) \in \cal S^{n+1}$.

In order to find such $\cal S^{n+1}$, let $\cal D^{n+1}$ be the Borel subset of $[X]^{< \infty} \times [\cal G]^{n+1}$ consisting of those separations nested with $\cal S_n$ and such that their $\Psi$-image is $n+1$-thin. The Borel graph $\cal H_{n+1}$ on $\cal D_{n+1}$ defined by two separations being non-nested is locally finite. Indeed, for two such separations $(S,B)$ and $(S',B')$ to be not nested, we have that $(S',B')$ must separate at least two vertices in $S$. However, by Proposition \ref{prop:locfin} there are only finitely many such possible separations. Therefore, there exists a Borel kernel $\cal S^{n+1}$ for $\cal H_{n+1}$, which has the desired properties.

Let $\cal S:= \cal S^{\kappa}$. By Proposition \ref{prop:sep}, since $\Psi(\cal S)$ generates $\cal B_{\kappa} \cal G$ and components of $\cal G$ are $\kappa$-accessible, $\cal S$ has the required property: for any two ends $\omega,\omega'\in \ends (\cal G)$ defined in the same $\cal G$-component there exists $(S,B)\in \cal S$ separating them.

Let  $Q$ be the levels equivalence relation associated to $\cal S$ and $Z$ its vertex space as constructed in the previous section. Let $\cal L$ be the Borel graph on $Z$ defined by letting $(x,t) \sim_{\cal L} (y,t')$ if and only if $x \sim_{\cal G} y$ and $t = t'$. The Borel graph $\cal L$ may not yield the equivalence relation $Q$ whilst the torsos of the tree decomposition, which may be constructed as in Example \ref{ex:Tuttemes}, may not be minors of $\cal G$. Hence, we need to add edges to $\cal L$ in a slightly more cautious way in order to conclude the proof. 

Let us fix a Borel linear order on $X$. Hence, for each separation $(S,B)\in \cal S$ we may enumerate $S = \{s_1,\dots,s_{|S|}\}$ using the Borel linear order. Let $(A_1,A_2)$ denote the sides of $(S,B)$. Let $\alpha^{(S,B)}_{s_1}$ denote the connected subgraph of $\cal G$ obtained as the connected component of $\cal G[A_2]\backslash\{s_2,\dots,s_{|S|}\}$ containing $s_1$. For $i\in \{1,\dots,|S|\}$, let $\alpha^{(S,B)}_{s_i}$ denote the connected subgraph of $\cal G$ obtained as the connected component of $$
\cal G[A_2] \backslash [(\cup_{j< i} U_j) \backslash (\cup_{i< l\leq |S|} \{u_l\}),
$$ 
containing $s_i$. For each $(x,t) \in Z$ we let $\alpha_{(x,t)}$ denote the union of all the $\alpha_x^{(S,B)}$ where $(S,B) \in t$ and $x\in S$.

We let $\cal A \subset Z\times Z$ be the Borel set of edges of the form $(x,t) \sim_{\cal A} (y,t')$ if and only if $t = t'$ and there exists $(S,B) \in t$ with  $x,y\in S$ such that there exists an edge $e\in \cal G$ with one end-vertex in $\alpha_{(x,t)}$ and the other end-vertex in $\alpha_{(y,t)}$. We claim that components of $\Rel (\cal L\cup \cal A) = Q$ and that $\cal L\cup \cal A\in \cal B(\cal F_1)$.

Let us first show that for every path $P \subset \cal G$ with end-vertices $u,v$ in $V_t$ there exists a path $\tilde P \subset \cal L \cup \cal A$ with end-vertices $(u,t)$ and $(v,t)$. This proves that $\Rel (\cal L\cup \cal A) = Q$. Let $P$ be such a path with end-vertices in $V_t$. The path $P$ is a concatenation of maximal subpaths in $\cal L := \cal G [V_t]$ and maximal subpaths with edges not in $\cal L$. It thus suffices to prove our claim if $P$ is a path with both end-vertices $u,v$ in $V_t$, but all internal vertices not in $V_t$. By Proposition \ref{prop:CHM}, there exists $(S,B)\in \cal S$ such that $u,v \in  S$ and $V(P) \subset V(A_2) \cup \{u,v\}$, where $(A_1,A_2)$ are the sides of $(S,B)$. We then let $\tilde P$ be the path with edges $((s,t),(s',t))\in \cal A$ if and only if $s,s'\in S$ and there exists an edge in $P$ with one end-vertex in $\alpha_s^{(S,B)}$ and the other in $\alpha_{s'}^{(S,B)}$. By construction the path $\tilde P$ has the desired properties.

It is only left to show that $\cal L\cup \cal A\in \Bor(\cal F_1)$. This requires proving that $\cal L \cup \cal A \in \Bor (\cal F)$ and that $\cal L \cup\cal A$ has at most one-ended components. To show that $\cal L \cup \cal A \in \Bor (\cal F)$ we observe that, by construction, the components of $\cal L\cup \cal A$ are minors of components of $\cal G$, which are in $\cal F$. Moreover, since each element of $X$ is contained in the adhesion set of at most finitely many separations of $\cal S$ by an argument along the line of that in Proposition \ref{prop:locfin}, then $\cal L\cup \cal A$ is locally finite. Therefore, components of $\cal L \cup \cal A$ are locally finite minors of graphs in $\cal F$, hence are in $\cal F$ by hypothesis. We deduce that $\cal L \cup \cal A \in \Bor (\cal F)$.

To prove one-endedness, let $R_1$ and $R_2$ be two rays in the same $\cal L \cup \cal A$-component. Let $P_1$ and $P_2$ be two rays in $\cal G$ intersecting the rays $R_1$ and $R_2$ infinitely often. The rays $P_1$ and $P_2$ must be equivalent. Otherwise, by the accessibility hypothesis and Proposition \ref{prop:sep} there would exist a separation in $\cal S$ separating tails of $P_1$ and $P_2$. This would imply that tails of $R_1$ and $R_2$ cannot be contained in the same level.

Let $U\subset Z$ be a finite set of vertices in the $Q$-class supporting $R_1$ and $R_2$. Then, by the above, there exists a path $J$ in $\cal G$ with end-vertices $x,y\in X$ such that $(x,t)\in V(R_1)$ and $(y,t) \in V(R_2)$ and such that $V(J)$ does not intersect $p(U)$ or any adhesion set $S$ of a separation $(S,B)\in \cal S$ such that $p(U)\cap S \neq \emptyset$, where $p\colon Z \rightarrow X$ is as in Proposition \ref{prop:Z}. Indeed, observe that there are finitely many such separations $(S,B)$, and so existence of $J$ follows from equivalence of $P_1$ and $P_2$. Then, the path $\tilde J$ constructed as above has end-vertices in $R_1$ and $R_2$, but contains no vertex in $U$. Hence, we conclude that $R_1$ and $R_2$ are equivalent, concluding the proof.
\end{proof}

Before we address the applications of the two theorems above, let us pose a question.

\begin{question*}[Question \ref{rem:q2}]\label{q2}
For a locally finite borel graph $\cal G$, does there exist a nested Borel  separation system with property ($\ast$) separating any two ends in the same component?
\end{question*}

\begin{remark}\label{rem:q2}
A positive answer to Question \ref{rem:q2} would imply the following. Let $\cal G$ be a Borel graph and $\cal S$ a nested Borel separation system with property ($\ast$) separating any two ends in the same component. Running the construction of the proof of Theorem \ref{tma:acces} using $\cal S$, would yield a Borel graph $\cal H$ such that $\Rel (\cal H) = Q$, components of $\cal H$ are at most one-ended, and $\cal H$-components are $\cal G$-minors. This would imply that, for a minor-closed family of locally finite connected graphs $\cal F$, every Borel graph in $\Bor (\cal F)$ is $\cal F_1$-accessible.

Hence, a positive answer to Question \ref{q2} would imply that locally finite Borel graphs with no thick ends are Borel treeable from Theorem \ref{tma:thin}. Similarly, it would imply from Theorem \ref{tma:planar} that all locally finite Borel graphs with planar components are measure treeable.
\end{remark}

As a first corollary of the above theorems, we get our first application below. Let $\Thin$ denote the class of connected locally finite graphs with no thick ends. We have the following theorem.

\begin{theorem}\label{tma:thin}
Borel $\Thin_1$-accessible graphs are Borel treeable. 
\end{theorem}

Before we address the proof of the theorem, let us point out the following more concrete corollary of Theorem \ref{tma:thin}.

\begin{corollary}\label{cor:bddtw}
The following classes of Borel graphs are Borel $\Thin_1$-accessible, hence Borel treeable:
\part Borel graphs with no thick ends and accessible components. 

\part Uniformly locally finite Borel graphs with components of bounded tree-width.

\trap 
\end{corollary}

\begin{proof}
By Halin's Theorem \cite{Hal}, the family $\Thin$ is minor-closed. Therefore, part (a) follows directly from Theorem \ref{tma:acces}. In order to prove part (b), note that graphs of bounded tree-width have no thick ends. Hence, it is enough to show that uniformly locally finite graphs of bounded tree-width are accessible. But a graph $G$ with tree-width at most $k$ is $k$-vertex-accessible, so $kD$-edge-accessible, where $D$ is the maximum degree.
\end{proof}

\begin{remark}
In order to extend the above corollary to arbitrary locally finite Borel graphs with components of bounded tree-width it would suffice to either show that locally finite connected graphs with bounded tree-width are accessible or to provide a positive answer to Question \ref{q2}.
\end{remark}

We conclude this section with the proof of Theorem \ref{tma:thin}.

\begin{proof}[Proof of Theorem \ref{tma:thin}]

Let $\cal F \subset \cal G$ be an aperiodic Borel subforest of $\cal G$, existent by Proposition \ref{prop:forex}. Each component of $\cal F$ contains a ray in $\cal G$ representing the unique end of its component. Hence, for ends of $\cal G$ being thin there are finitely many such possible disjoint rays. We have that $\Rel (\cal F)$ has finite index in $\Rel (\cal G)$, and that $\cal F$-components have finitely many ends. By \cite[Theorems D and E]{Mil} and \cite[Section 8]{DJK} we have that $\cal F$ is a hyperfinite finite index subequivalence relation of $\cal G$ a.s. We deduce that $\Rel (\cal G)$ is $\mu$-hyperfinite by \cite[Proposition 1.3.vii]{JKL}, so $\mu$-treeable.
\end{proof}

Observe that the above argument yields the following corollary.

\begin{corollary}\label{cor:hyper}
Let $\cal G$ be an at most one-ended locally finite Borel graph. If $\cal G$ has no thick ends, then $\Rel (\cal G)$ is hyperfinite.
\end{corollary}


\section{Applications to planar graphs}

The goal of this section is to deduce that locally finite Borel graphs with planar accessible components are measure treeable, thus generalising results in \cite{CGMT} and \cite{Timar}. We will achieve this applying Theorem \ref{tma:acces}. Hence, letting $\cal P$ be the family of Borel graphs with planar components, our goal is to prove the following.

\begin{theorem}\label{tma:one-end-plan}
Let $\cal G$ be a locally finite Borel graph with planar components. If $\cal G$ is at most one-ended, then $\cal G$ is measure treeable.
\end{theorem}

By Theorem \ref{tma:acces}, the above will imply the following. 

\begin{theorem}\label{tma:planar}
Borel $\cal P_1$-accessible Borel graphs are measure treeable. In particular, Borel graphs with accessible planar components are measure treeable.
\end{theorem}

As we shall discuss in Example \ref{ex:freepakosanz}, the family of $\cal P_1$-accessible Borel graphs contains unimodular random graphs which are not accessible a.s. Therefore, the above theorem applies beyond Borel graphs with accessible components.

Our strategy to prove Theorem \ref{tma:one-end-plan} is to apply the treeability result for Borel graphs with a Borel 2-basis \cite[Theorem 3.6]{CGMT}. In order to apply this result, we first need to induce a Borel 2-basis structure on an arbitrary Borel graph with planar components or some suitable extension. In this direction, we introduce Borel spherical graphs below. Roughly speaking, a Borel spherical graph is a Borel graph with planar components together with a Borel structure encoding the combinatorial information of an embedding in the 2-dimensional sphere $\SS$. 

\begin{definition}\label{def:borsph}
A locally finite Borel graph $\cal G$ on a standard Borel space $X$ is \emph{Borel spherical} if $\cal G$ is equipped with a Borel map $\omega\colon X \rightarrow \Cyclic (X)$ such that for every $x\in X$, $\om(x)$ is a cyclic order on $N_{\cal G}(x)$, and  $\{\om (x)\}_{x \in V(\cal G_x)}$ is a spherical rotation system on $\cal G_{x}$.  Such a map $\omega$ is called a \emph{Borel rotation system}.
\end{definition}

In the next example we show that Borel graphs with planar components are not necessarilly  Borel spherical.
 
 \begin{example}\label{ex:adams}
Let $\cal L$ be a Borel forest of lines on a standard Borel space $X$ which is not generated by a Borel $\ZZ$-action. Such an example was first proposed by Adams \cite{Adams90}. We consider the graph $\cal G$ on $X \times \ZZ_4$ given by $(x,i) \sim (y,j)$ if and only if $x \sim_{\cal L} y$ and $i = j$, or, $x=y$ and $i = j\pm 1 \pmod 4$. Components of $\cal G$ are  isomorphic to the graph of Example \ref{ex:miticu1}, so planar. 

However, $\cal G$ is not Borel spherical. Indeed, we show that if $\omega$ is a Borel rotation system on $\cal G$, then $\cal L$ admits a Borel $\ZZ$-action. For a vertex $x \in X$, let $C_x$ denote the oriented cycle $(x,0) \rightarrow (x,1) \rightarrow \dots \rightarrow (x,0)$. Fix a finite connected subgraph $L $ of $\cal L$. Consider an embedding $\alpha : \cal G [V(L) \times \ZZ_4] \rightarrow \RR^2$ such that the restriction of $\omega$ is induced by $\alpha$ and such that for every $x \in V(L)$, the directed cycle $\alpha (C_x)$ winds once around the origin in a counterclockwise fashion. This embedding induces a linear order $\leq_L$ on $V(L)$ by letting $x \leq y$ if and only if $\alpha (C_x)$ lies inside $\alpha (C_y)$. Note that the order on $L$ is independent of the choice of $\alpha$ after the restriction on the cycles $\alpha (C_x)$ of winding counterclockwise around the origin.

Using Whitney's Theorem, we see that if $L \subset K$ are both finite connected subgraphs of $\cal L$, then $\leq_L$ coincides with the restriction of $ \leq_K $ to $V(L)$. These orders assign an orientation in a Borel fashion to the orbits of $\cal L$, hence inducing a Borel $\ZZ$-action, contradicting the undirectability of $\cal L$.
\end{example}

Despite of the above example, one can prove that Borel graphs with planar components admit Borel spherical extensions. 

\begin{prop}\label{prop:ext}
Let $\cal G$ be a locally finite Borel graph with 3-connected planar components. Then, there exist a Borel spherical graph $\cal H$ and a 2-to-1 extension $\cal H \rightarrow \cal G$. 
\end{prop}

We construct the extension by exhibiting an explicit Borel cocycle. Let us first prove the following lemma. 

\begin{lemma}\label{prop:rota}
Let $\cal G$ be a locally finite Borel graph with 3-connected planar components on a standard Borel space $X$. Then, there exists a Borel map $\omega\colon X\rightarrow \Cyclic (X)$ such that for every $x\in X$, the cyclic order $\omega(x)$ is a rotation at $x$ induced by a spherical embedding of $\cal G_x$.
\end{lemma}

\begin{proof}
For each $r\in \NN$, let $Y_r \subset X \times \Cyclic (X)$ be the Borel subset consisting of those $(x,\omega) \in  X \times \Cyclic (X)$ such that $\omega$ is a rotation at the vertex $x$ of $\cal G$ induced by a spherical embedding of the ball $B_r (x) \subset \cal G$. It follows after Lemma \ref{lem:whitcore} that the Borel set $Y = \bigcap_{r\in \NN} Y_r$ is the Borel subset of $X \times \Cyclic (X)$ consisting of those $(x,\omega)\in  X \times \Cyclic (X)$ such that $\omega$ is induced by a spherical embedding of $\cal G_x$.  

By Imrich's Theorem and connectivity, for each $x\in X$, the section $Y_x \colon=\{\omega \in \Cyclic (X)\colon (x,\omega) \in Y \}$ has exactly two elements. Therefore, $Y$ admits a Borel uniformization by the Lusin-Novikov theorem. Any such uniformization $\omega$ is the map we are looking for.
\end{proof}

\begin{proof}[Proof of Proposition \ref{prop:ext}]

We define a Borel cocycle $c\colon \cal G \rightarrow \mathbb{Z}_2$ by letting $c_{(x,y)} = 0$ for an edge $(x,y)\in \cal G$ if $\om (x)$ and $\om (y)$ are induced by the same spherical embedding of $\cal G_x$ and  $c_{(x,y)} = 1$ otherwise.

We show that $c$ is indeed Borel. To do this we prove that we only need a finite neighbourhood of $(x,y)$ and $\om$ to compute $c_{(x,y)}$. Let $e=(x,y) \in \cal G$. By Menger's Theorem and 3-connectivity of components, the vertices $x$ and $y$ lie in two cycles $A$ and $B$ intersecting exactly at the vertices $x,y$ and the edge $e$. For $\cal G$ being locally finite, such $A$ and $B$ can be chosen in a Borel way. Indeed, the cycles $A$ and $B$ can be taken to be the two minimum elements with respect to a Borel $\NN$-ordering of the set of cycles containing $(x,y)$. Such Borel $\NN$-ordering can be obtained using the cycle length and a Borel linear order on the vertex space. Let $A_x$ and $B_x$ be the vertices in $A$ and $B$ respectively adjacent to $x$ which are not $y$. Similarly, consider $A_y$ and $B_y$. 

Note that, by construction, the cycles $A$ and $B$ are non-crossing in every spherical embedding. Let us suppose that the rotations $\om (x)$ and $\om (y)$ are induced by the same spherical embedding of $\cal G_x$. If $\om (x)$  induces a cyclic order $[y,A_x,B_x]$, or $[y,B_x, A_x]$, then $\om(y)$ must order $[x,B_y,A_y]$, or $[x, A_y, B_y]$ respectively. Indeed, it is easy to see that, otherwise, the cycle $B$ would have vertices on both cyclesides of $A$, implying that $A$ and $B$ cross, what we have noted not to be possible. 

It also follows from the above argument that if $\om (x)$ orders $[y,A_x,B_x]$, or $[y,B_x, A_x]$, and $\om(y)$ orders $[x,A_y,B_y]$, or $[x, B_y, A_y]$ respectively, then  $\om (x)$ and $\om (y)$ cannot be induced by the same spherical embedding. In this latter case, it follows from Imrich's Theorem that it is $\om (x)$ and $\om (y)^-$ that are induced by the same spherical embedding of $\cal G_x$.  Therefore, as claimed, the cocycle $c$ is Borel as the value $c_{(x,y)}$ is decided by comparing the cyclic ordering of $y, A_x$, and $B_x$ induced by $\om (x)$ and the cyclic ordering of $x, A_y$, and $B_y$ induced by $\om (y)$.

The Borel spherical extension $\cal H$ is the Borel graph on $X\times \mathbb Z_2$, where $X$ is the vertex space of $\cal G$, defined by setting an edge between $(x, i)$ and $(y,j)$ if and only if $(x,y) \in \cal G$ and $i = c_{(x,y)} + j \pmod 2$. By construction of the cocycle, the map $\omega$ on $\cal H$ defined by $\om (x, 0) = \om (x)$ and $\om (x,1) = \om (x)^-$ is a Borel spherical rotation.
\end{proof}

\begin{corollary}\label{cor:dosbase}
Let $\cal G$ be a locally finite Borel graph with 3-connected planar components admitting a 2-basis. Then, $\cal G$ admits a Borel 2-basis.
\end{corollary}

\begin{proof}
Let $\cal H$ be the auxiliary extension of the previous proposition. We may compute a Borel 2-basis $\cal B$ for $\cal H$ as those facial cycles of the embeddings specified by the Borel rotation system. By Imrich's theorem, the facial cycles are independent of the embedding. The image of $\cal B$ under the extension map is hence a Borel 2-basis for $\cal G$.
\end{proof}

\begin{proof}[Proof of Theorem \ref{tma:one-end-plan}]
Let $\cal G$ be a Borel graph with one-ended planar components. The torsos of its Tutte decomposition as in Example \ref{ex:Tuttemes} are also planar by Lemma \ref{lem:margullo}. These torsos are at most one-ended by an argument along the lines of that in Theorem \ref{tma:Faccessible} and locally finite by an argument as in \cite[Proposition 4.2]{Thom93}. Hence, by Theorem \ref{tma:maintech} we may assume without loss of generality that $\cal G$ has one-ended 3-connected planar components. However, planar one-ended graphs admit a 2-basis by Thomassen's Theorem, as they admit an embedding in $\cal S_1$ with a single accumulation point. Therefore, by Corollary \ref{cor:dosbase} the Borel graph $\cal G$ admits a Borel 2-basis and is thus measure treeable by \cite[Theorem 3.6]{CGMT}.
\end{proof}

Observe that there are examples of $\cal P_1$-accessible Borel graphs -- actually graphings -- whose components are not accessible. The author is not aware of a Borel graph with planar components which is not $\cal P_1$-accessible.

Note that this example provides an example of unimodular random graph which is planar but not accessible almost surely. This example lies in contrast with the quasi-transitive setting, where planar graphs are known to be accessible \cite{Ham18}. 

\begin{example}\label{ex:freepakosanz}
Let $\FF_2 = \left< a, b\right>$ denote the free group on two generators. Consider a p.m.p.~ergodic free action of $\FF_2$ on a standard probability space $(X, \mu)$ with a Borel partition $X = \bigsqcup_{n\geq 4} X_n$ such that $b \cdot X_n = X_n$ with $0 < \mu (X_n) < 1$ for every $n \geq 4$, and $\sum_{n \geq 4} n \mu (X_n) < \infty$. Such $X$ may be constructed from a p.m.p.~ergodic free action $\FF_2 \actson (Y,\nu)$  by considering inverse limits of diagonal actions $\FF_2 \actson (Y \times \ZZ_n , \nu \otimes \rho_n)$ where $\rho_n$ is the normalized counting measure on $\ZZ_n$ and $\FF_2 \actson \ZZ_n$ is defined by $a \cdot i = i+1$ and $b \cdot i = i$ for every $ i \in \ZZ_n$.

We let $\tilde X = \bigsqcup_{n \geq 4} X_n \times \ZZ_n$ be equipped with the Borel finite measure $\tilde \mu = \sum_{n \geq 4} \mu|_{X_n} \otimes \lambda_n$, where $\lambda_n$ is the counting measure on $\ZZ_n$. For each $x \in X$, we let $n_x$ denote the unique $n\in \NN$ such that $x \in X_n$. Consider the Borel graph $\cal G$ given by the edges
\begin{itemize}
\item $(x,i) \sim (b \cdot x, i)$ for every $x \in X$ and $i \in \ZZ_{n_x}$,

\item $(x,i) \sim (x, i+1)$ for every $x \in X$ and $i \in \ZZ_{n_x}$, 

\item $(x,j) \sim (a\cdot x, n_{a\cdot x} -j)$ for every $x \in X$ and $j \in \{1,2\}$, and

\item $(x,j) \sim (b^{-1} a b \cdot x , n_{b^{-1} a b\cdot x}-j)$ for every $x \in X$ and $j \in \{1,2\}$,
\end{itemize}
where all the sums in the indices are considered $\pmod n_x$. By construction, the finite measure $\tilde\mu$ is preserved by $\cal G$, so the pair $(\cal G, \tilde \mu / \tilde \mu (\tilde X))$ is a graphing. 

Components of $\cal G$ are spherical. To prove this let us start by denoting by $U_x$, for each $x \in X$, the subgraph of $\cal G$ defined by $\cal G [\{(b^n\cdot x,j): n \in \ZZ, j \in \ZZ_{n(x)}\}]$. If $n(x) = 4$, the graph $U_x$ is the graph in Example \ref{ex:miticu1}. We also have that $U_x =U_{b\cdot x}$. To embed a given component in $\RR^2$, one may start by choosing some $x \in X$ and embedding $U_x$ in $\RR^2 $. Graphs of the form $U_{ab^n \cdot x}$ for any $n \in \ZZ$ are then embedded in faces of $U_x$. Proceeding iteratively, we obtain an embedding in $\RR^2$ of the component of $x$.

Each such graph $U_x$ contains two ends of its connected component. It is easy to see that, in order to separate these two ends, one needs to remove at least $n_x$ vertices, what can be attained by removing one of the $\ZZ_{n_x}$ cycles. Since $n_x$ is unbounded on each component, by ergodicity, it follows that components of $\cal G$ are not accessible a.s. 
\end{example}

\section{Property (T) equivalence relations have no planar graphings}

In this last section we prove approximability of p.m.p.~countable Borel equivalence relations admitting a multi-ended graphing. In particular, this implies that p.m.p. countable Borel equivalence relations with measure property (T) admit no graphing with planar components almost surely. This is Theorem \ref{tma:noT}.

\begin{prop}\label{prop:multi}
Let $(R, \mu)$ be a p.m.p.~countable Borel equivalence relation with a multi-ended a.s.~graphing $\cal G$. Then $R$ is approximable.
\end{prop}

\begin{proof}
Let $X$ denote the vertex space of $\cal G$. Arguing along the lines of the proof of Theorem \ref{tma:acces} we may construct a $\cal G$-invariant Borel function $f\colon X\rightarrow \NN$ such that $f(x)$ is the least number of edges required to separate two ends in the component of $x$. We have that $f(x)$ is well-defined a.s.~for $\cal G$ being multi-ended a.s.

Again using Dicks-Dunwoody's argument \cite[Section II.2]{DD} as in the proof of Theorem \ref{tma:acces}, we find a nested Borel separation system $\cal S$ such that, possibly upon passing to a Borel subset, every separation of $\cal S$ separates two ends of its component.
By thinness and an argument along the lines of Proposition \ref{prop:locfin} we have that, moreover, every $x\in X$ lies in at most finitely many adhesion sets of separations in $\cal S$. If required, by a measure exhaustion argument we may assume that $\cal S$ has elements in almost every orbit.

Using the Marker Lemma \cite[Lemma 6.7]{KM}, we find a sequence $\cal S_n$ of Borel subsets of $\cal S$ with elements in almost every orbit such that $\cal S_n \supset \cal S_{n+1}$ for every $n\in \NN$ and $\bigcap_{n\in \NN} \cal S_n = \emptyset$. We let $R_n \leq R$ be the Borel subequivalence relation defined on $X$ by $x \sim_{R_n} y$ if and only if both $x$ and $y$ lie in a unique part of the tree decomposition induced by $\cal S_n$ in their component. 

By Lemma \ref{prop:locfin}, and for every vertex lying in finitely many adhesion sets of separations of $\cal S$, we have that, given any $x \sim_{\Rel (\cal G)} y$, then $x\sim_{R_n} y$ for $n$ sufficiently large. Hence, to conclude the proof it suffices to show that for every $R_n$ and Borel subset $A\subset X$ such that $\Rel (\cal G)|_A = R_n|_A$ we have that $\mu (A) = 0$. 

The Borel subset $A$ selects a unique part of the tree decomposition induced by $\cal S_n$ in each component intersected by $A$. Let $\cal P \subset \cal S_n$ consist of those separations defining a part selected by $A$. For each $(S,B) \in \cal P$, we may choose a vertex $u_{(S,B)} \in S$ using a Borel linear order. Using again a Borel linear order and its induced lexicographic order on paths, we choose in a Borel way a ray $R_{(S,B)}$ with one initial vertex $u_{(S,B)}$ such that $R_{(S,B)}$ a tail of $R_{(S,B)}$ is entirely contained in the side of $(S,B)$ not containing the part selected by $A$ in its component. 

By the definition of tree decomposition, the rays $R_{(S,B)}$ have disjoint tails for each $(S,B) \in \cal P$. Moreover, these rays are indeed infinite for our assumption that each $(S,B)$ separates at least two ends, so has two infinite sides. The existence of the $P_{(S,B)}$ implies that $\mu (A) = 0$ for $\mu$ being preserved by $R$, concluding the proof.
\end{proof}

The above Proposition implies, by \cite{pichot}, the following generalization of the well-known fact that countable groups with property (T) are at most one-ended.

\begin{prop}
Let $(R,\mu)$ be a p.m.p.~Borel equivalence relation with measured property (T). Then $R$ does not admit an a.s.~multi-ended graphing.
\end{prop}

This proposition implies the desired conclusion.

\begin{theorem}\label{tma:noT}
Let $(R,\mu)$ be a p.m.p.~Borel equivalence relation with measured property (T). Then $R$ does not admit a locally finite graphing with planar components a.s.
\end{theorem}

\begin{proof}
If $R$ admitted a locally finite graphing with planar components a.s., then such components would be one-ended by the previous proposition. However, then $R$ would be $\mu$-treeable by Theorem \ref{tma:one-end-plan}, so $(R,\mu)$ cannot have property (T). 
\end{proof}

\section*{Acknowledgements}

The author would like to thank his advisor \L{}ukasz Grabowski for his guidance and advice. The author would also like to thank Anush Tserunyan for her valuable comments on some constructions in the paper, and Robin Tucker-Drob for pointing out Dicks-Dunwoody's argument.

\end{document}